\documentclass[a4paper,11pt]{amsart}

\title[Frayed Demazure weaves]{Frayed Demazure weaves for Poisson-compatible cluster structures on Bott--Samelson charts}
\author{Jon Cheah}
\address{
	Department of Mathematics   \\
	The University of Hong Kong \\
	Pokfulam Road               \\
	Hong Kong}
\email{joncheah@connect.hku.hk}

\date{\today}  

\calclayout

\usepackage{amsmath}
\usepackage{amsfonts}
\usepackage{mathtools}
\usepackage{bbm}
\usepackage{amssymb}
\usepackage[dvipsnames]{xcolor}
\usepackage{graphicx}
\usepackage{array} 
\usepackage{multirow} 
\usepackage{bm}
\usepackage[shortlabels]{enumitem}
\usepackage{mathrsfs}
\usepackage{etoolbox}
\usepackage{pdfpages}

\usepackage{quiver}
\usepackage{tikz}
\usetikzlibrary{calc}
\usepackage{tikz}
\usetikzlibrary{decorations.pathreplacing,angles,quotes,patterns,patterns.meta}
\usepackage{tikz-cd}
\usepackage{subcaption}
\usepackage[numbers]{natbib}
\usepackage{cite}
\usepackage{hyperref}
\usepackage{extarrows}
\usepackage{mathdots}

\usepackage{amsthm} 
\theoremstyle{plain}
\newtheorem{thm}{Theorem}[section] 
\newtheorem{lem}[thm]{Lemma} 
\newtheorem{prop}[thm]{Proposition} 
\newtheorem{cor}[thm]{Corollary}

\theoremstyle{definition}
\newtheorem{defn}[thm]{Definition}
\newtheorem{eg}[thm]{Example}

\theoremstyle{remark}
\newtheorem{rem}[thm]{Remark}

\newcommand{\Z}{\mathbb{Z}}
\newcommand{\Q}{\mathbb{Q}}
\newcommand{\C}{\mathbb{C}}

\renewcommand{\P}{\mathbb{P}}
\newcommand{\proj}{\mathbb{P}}

\newcommand{\blank}{{-}}
\newcommand{\largewedge}{\mbox{\Large $\wedge$}}

\newcommand{\ep}{\varepsilon}

\newcommand{\cA}{\mathcal{A}}
\newcommand{\cB}{\mathcal{B}}

\newcommand{\cO}{\mathcal{O}}

\newcommand{\seed}{\mathbf{s}}

\definecolor{S1}{RGB}{10,70,150}
\definecolor{S2}{RGB}{255,102,102}
\definecolor{S3}{RGB}{78,174,30}

\usepackage{url}

\begin{document}
	\begin{abstract}
		Demazure weaves are combinatorial representations of maps between Bott--Samelson cells and have been used to construct cluster structures on braid varieties. We show the compatibility of these maps and the resulting cluster structures with the standard Poisson structure on the Bott--Samelson variety. Adding frayed strands to Demazure weaves, we further construct Poisson-compatible cluster structures on other affine charts of the Bott--Samelson variety in a manner that transition functions across charts become rational quasi-cluster. The mutation sequences we construct for these quasi-cluster morphisms are closely related to those of M\'enard for open Richardson seeds. 
	\end{abstract}
	\maketitle
	\setcounter{tocdepth}{1}
	\tableofcontents
	
	\subsection{Scientific context}
	
	Open Richardson varieties were introduced by Kazhdan--Lusztig, and are useful in studying both the total positivity \citep{MR04} and Poisson geometry \citep{GY09} of the flag variety. Cluster algebras were introduced by Fomin and Zelevinsky \citep{FZ02} as a means to study canonical bases in quantum groups and total positivity in algebraic varieties. It is a natural question whether varieties arising from Lie theory possess cluster structures compatible with total positivity and Poisson geometry. 
	
	For a simply-laced algebraic group $G$, Leclerc \citep{Lec16} used modules of the preprojective algebra to construct a cluster subalgebra in the coordinate ring of open Richardson varieties in the flag variety. It was conjectured that this subalgebra is equal to the whole coordinate ring and proved in some special cases. The conjecture was refined by M\'enard \citep{M22} who provided a candidate initial seed, and it was shown in \citep{CK22} that this seed defines an upper cluster algebra structure the coordinate ring. M\'enard's algorithm was recently used in \citep{BY25} to construct an upper cluster structure on Kac--Moody open Richardson varieties.
	
	When $G$ is finite-type, open Richardson varieties are special cases of braid varieties. Two separate groups independently resolved Leclerc's conjecture by constructing cluster algebra structures on braid varieties. Casals--Gorsky--Gorsky--Le--Shen--Simental \citep{CGGLSS25} used the combinatorics of \emph{Demazure weaves} and Galashin--Lam--Sherman-Bennett--Speyer \citep{GLSBS24,GLSB25} used 3D plabic graphs and Deodhar geometry.	
	These two cluster structures coincide. A comparison of their tori, cluster variables, and Weil--Petersson forms is presented in \citep{CGGSSBS25}.
	
	Demazure weave combinatorics were a main inspiration for this paper. Weaves were introduced in \citep{CZ22} to study Lagrangian filings of Legendrian links. They returned in \citep{CGGS24} and were used to construct many decompositions of braid varieties. A subclass of weaves, called \emph{Demazure weaves}, correspond to those whose unique strata of maximal dimension are isomorphic to an algebraic torus, and these form the initial cluster tori in \citep{CGGLSS25}.

	\subsection{Poisson compatibility}
	A cluster structure and a Poisson structure on an algebra are \emph{compatible} \citep{GSV03,GSV10} if every cluster is comprised of mutually log-canonical elements.  In \citep{GSV10}, it was shown that the cluster algebra structure in \citep{BFZ05} constructed on double Bruhat cells in semisimple algebraic groups is compatible with (the restriction of) the \emph{standard Poisson structure}. Previous works \citep{SW21,CGGLSS25} also construct cluster Poisson structures on Bott--Samelson cells and braid varieties respectively. However, the Poisson structure arising from this need not be the one induced by the standard Poisson structure on the algebraic group $G$, see \citep[Remark 1.11]{SW21}.

	\subsection{Bott--Samelson charts}
	Braid varieties and open Richardson varieties are locally-closed subvarieties of an iterated $\P^1$-fibration called a \emph{Bott--Samelson variety}.
	A Bott--Samelson variety $\mathring{Z}_{\underline{w}}$ has $2^{\ell(\underline{w})}$ natural affine coordinate charts indexed by subexpressions of the word $\underline{w}$.
	Elek and Lu \citep{EL21} computed the standard Poisson structure on $Z_{\underline{w}}$ in the natural \emph{Bott--Samelson coordinates} of these charts. It was shown that these coordinates form an iterated $T$-Poisson Ore extension, and that for particular subexpressions, they are in fact symmetric Poisson CGL extensions. A result of Goodearl and Yakimov \citep{GY18} says that when this occurs, there exist cluster structures compatible with the Poisson structure.

	\subsection{Structure of the paper}
	Sections \ref{sec:clusterbackground} and \ref{sec:BSvarieties} set up the background on cluster algebra theory and the algebraic varieties that we study.
	In Section \ref{sec:classical weaves}, we introduce Demazure weaves from \citep{CGGLSS25} and their associated toric charts. We show that the component maps associated to weave vertices are Poisson maps, and prove in Theorem \ref{thm:weavechartsarelogcan} that the corresponding weave tori on braid varieties (and Bott--Samelson cells) are log-canonical in the standard Poisson structure. By considering the initial torus inside the opposite Bott--Samelson chart,
	we show the compatibility of the Poisson coefficient matrix and the exchange matrix in Corollary \ref{cor:ClusterPoissonCompatible}.
	In Section \ref{sec:weavesfornonfullcharts}, we modify weave combinatorics for the other Bott--Samelson charts with vertices corresponding to natural Poisson maps between Bott--Samelson charts. These allow us to pull back seeds from the supported subword and give a Poisson compatible cluster structure on $\cO^\gamma$ in Theorem \ref{thm:BSchartshaveclusterstructure}.
	In Section \ref{sec:clustertransitions}, we show that the constructed cluster algebra structure on the chart $\cO^\gamma$ is compatible with the existing one on $\cO^{\underline{w}}$ in the sense that the transition functions become rational quasi-cluster. The initial seed $\seed(\gamma)$ on $\cO^\gamma$ is quasi-equivalent to a freezing of a specific seed $\overrightarrow{\mu_{\gamma}}(\seed(\underline{w}))$ mutation equivalent to the seed $\seed(\underline{w})$ on $\cO^{\underline{w}}$.
	We describe $\overrightarrow{\mu_{\gamma}}(\seed(\underline{w}))$ as an iterated pullback of seeds. Proposition \ref{prop:MutSequenceAgrees} shows the mutation sequence $\overrightarrow{\mu_{\gamma}}$ is closely related to M\'enard's sequence \citep{M22,BY25}, and that the seeds for the chart $\cO^\gamma$ and the open Richardson $\mathring{Z}_{\underline{w}}^v$ are complementary in $\overrightarrow{\mu_{\gamma}}(\seed(\underline{w}))$.
	Running examples are provided throughout the work.

	\subsection{Acknowledgements}
	We thank Jiang-Hua Lu for helpful discussions on the standard Poisson structure and York Ye for explanations of M\'enard's mutation sequence.
	The author was supported by the HKU Presidential PhD Scholar Programme (HKU-PS).\\ No AI tools were used for the ideas, calculations, proofs, writing, or figures in this work.

	\section{Background on cluster algebras}\label{sec:clusterbackground}
	\subsection{Seeds and morphisms}
	Throughout this paper, let $J$ be a finite index set.
	\begin{defn}
		A \emph{seed} is a quadruple $\mathbf{s} = \bigl((A_j)_{j\in J}, J_{\mathrm{fr}},E,D\bigr)$ consisting of:
		\begin{itemize}
			\item A transcendental basis $(A_j)_{j\in J}$ of a field $K=\C(x_j)_{j\in J}$, called an \emph{extended cluster}.
			\item A subset $J_{\mathrm{fr}} \subset J$. The indices $j\in J_{\mathrm{fr}}$ and variables $(A_j)_{j\in J_{\mathrm{fr}}}$ are called \emph{frozen}. Indices and variables in the complement $J_{\mathrm{uf}} = J \setminus J_{\mathrm{fr}}$ are called \emph{mutable} or \emph{unfrozen}.
			\item An \emph{exchange matrix} $E = (\ep_{ij})_{i,j\in J}$, which is a skew-symmetrisable matrix with $\ep_{ij} \in \Q$ if both $i,j$ are frozen, and $\ep_{ij}\in \Z$ otherwise.
			\item A symmetriser for $D=\mathrm{diag}(d_1,\dots,d_\ell)$ such that $DE$ is skew-symmetric. We may assume that $D$ has coprime integer entries.
		\end{itemize}
	\end{defn}
	Given a seed $\mathbf{s}_0$ and mutable index $k$, its \emph{mutation} in direction $k$ gives a seed $\mathbf{s}'$ with the same symmetriser and frozen subset, but with exchange matrix $E' = (\ep'_{ij})_{i,j\in J}$ where
	$$\ep'_{ij} = \begin{cases}
		-\ep_{ij} \quad &\text{ if } k \in \{i,j\},\\
		\ep_{ij} + \ep_{ik}\ep_{kj}  \quad &\text{ if } k \notin \{i,j\}, \text{ and } \ep_{ik},\ep_{kj}>0,\\
		\ep_{ij} - \ep_{ik}\ep_{kj}  \quad &\text{ if } k \notin \{i,j\}, \text{ and }\ep_{ik},\ep_{kj}<0,\\
		\ep_{ij}, & \text{ else}.
	\end{cases}$$
	and whose variables are determined by $A_{i}' = A_i$, for  $i\neq k$, and  
	\begin{equation}\label{defn:mutation}
		A_k A_k' = {\prod_{\ep_{ik}<0}A_i^{-\ep_{ik}} + \prod_{\ep_{jk}>0} A_j^{\ep_{jk}}}.
	\end{equation}
	\begin{rem}\label{rem:rowvscolumnmutation}
		Some works, including \citep{SW21,CGGLSS25,BY25}, have the exponents in (\ref{defn:mutation}) read from the rows of the exchange matrix. The column convention used here aligns with \citep{FZ02,BZ05,FG09,GLS11}.
	\end{rem}
	
	We write $\mathbf{s} \sim\mathbf{s}_0$ when $\mathbf{s}$ can be obtained from $\mathbf{s}_0$ by a finite sequence of mutations. Given a seed $\mathbf{s}_0$, and a subset $\mathrm{inv}\subset J_{\mathrm{fr}}$ of frozen indices, we define the \emph{cluster algebra} $\mathcal{A}(\mathbf{s}_0,\mathrm{inv}) = \C[A_j]_{j\in J, \mathbf{s}\sim\mathbf{s}_0} [A_j^{-1}]_{j\in \mathrm{inv}}$ to be the the $\C$-subalgebra of $K$ generated by the (cluster and frozen) variables over all mutation-reachable seeds $\mathbf{s} \sim \mathbf{s}_0$, as well as the inverses of the variables indexed by $j\in \mathrm{inv}$.
	When $\mathrm{inv}=J_{\mathrm{fr}}$, we write $$\mathcal{A}(\mathbf{s}_0) := \mathcal{A}(\mathbf{s}_0,J_{\mathrm{fr}}) =  \bigl(\C[A_j^{\pm1}]_{j\in J_{\mathrm{fr}}}\bigr)[A_{j,\mathbf{s}}]_{i\in J_{\mathrm{uf}},\mathbf{s}\sim\mathbf{s}_0},$$
	so $\mathcal{A}(\mathbf{s}_0,\emptyset)\subseteq \mathcal{A}(\mathbf{s}_0,\mathrm{inv})\subseteq \mathcal{A}(\mathbf{s}_0)$.
	The \emph{upper cluster algebra} $\mathcal{U}(\mathbf{s}_0) = \bigcap_{\mathbf{s}\sim\mathbf{s}_0} \C[A_{j,\mathbf{s}}^{\pm1}]_{j\in J}$
	is the intersection of the Laurent polynomial rings over each cluster. If we further define 
	$$\mathcal{U}(\mathbf{s}_0,\mathrm{inv}) = \bigcap_{\mathbf{s}\sim\mathbf{s}_0} \C[A_{j,\mathbf{s}}^{\pm1}]_{j\in J_{\mathrm{uf}}}[A_j]_{j\in J_{\mathrm{fr}}}[A_j^{-1}]_{j\in \mathrm{inv}},$$
	then for any $\mathrm{inv}\subset J_{\mathrm{fr}}$, we have
	$\mathcal{U}(\mathbf{s}_0,\emptyset)\subseteq\mathcal{U}(\mathbf{s}_0,\mathrm{inv})\subseteq\mathcal{U}(\mathbf{s}_0)$.
	The celebrated \emph{Laurent phenomenon} \citep{FZ02} says that $\mathcal{A}(\mathbf{s}_0,\emptyset)  \subseteq \mathcal{U}(\mathbf{s}_0,\emptyset)$, and therefore $\mathcal{A}(\mathbf{s}_0,\mathrm{inv})  \subseteq \mathcal{U}(\mathbf{s}_0,\mathrm{inv})$ for any $\mathrm{inv}\subseteq J_{\mathrm{fr}}$. It is of general interest whether equality is achieved.
	A variety $X$ over $\C$ has a (resp. an \emph{upper}) \emph{cluster structure} if $\C[X]=\cA(\seed,\mathrm{inv})$ (resp. $\C[X]=\mathcal{U}(\seed, \mathrm{inv})$) as subalgebras of $\C(X)$ for some initial seed $\seed$ and inversion set $\mathrm{inv}\subseteq J_{\mathrm{fr}}$.
	
	Given a seed $\mathbf{s}$, its \emph{freezing at $i\in J_{\mathrm{uf}}$}, is the seed $\mathbf{s}' = \bigl((A_j)_{j\in J}, J_{\mathrm{fr}} \cup \{i\},B,D\bigr)$. Its \emph{deletion at $i\in J$} is the seed $\mathbf{s}^\dagger = \bigl((A_j)_{j\in J\setminus\{i\}}, J_{\mathrm{fr}} \setminus \{i\},(\ep_{jk})_{j,k\neq i},(d_j)_{j\neq i}\bigr)$.
	\begin{lem}[\protect{\citep[Proposition 3.1]{Mul13}}]
		If $\mathbf{s}'$ is obtained from $\mathbf{s}$ by freezing at $i\in J$, then we have
		$\cA(\mathbf{s}') \subseteq \cA(\mathbf{s}) [A_{i,\mathbf{s}}^{-1}]\subseteq \mathcal{U}(\mathbf{s}) [A_{i,\mathbf{s}}^{-1}]\subseteq \mathcal{U}(\mathbf{s}')$,
		as subalgebras of the ambient field $K$.
	\end{lem}
	
	\begin{lem}[\protect{\citep[Proposition 3.7]{Mul13}}]
		If $\mathbf{s}^\dagger$ is obtained from $\mathbf{s}$ by deleting a frozen index $i\in J_{\mathrm{fr}}$, we have natural isomorphisms
		$\cA(\mathbf{s})/(A_{i} -1) \cong \cA(\mathbf{s}^\dagger)$ and $\mathcal{U}(\mathbf{s})/(A_{i} -1) \cong \mathcal{U}(\mathbf{s}^\dagger)$.
	\end{lem}
	
	By construction, mutation equivalent seeds yield the same (upper) cluster algebra. We will also make use of the notion of \emph{quasi-equivalence} of seeds, first introduced in \citep{Fra16}, which is a more flexible criterion for when two seeds give related cluster algebras.
	
	\begin{defn}\label{defn:quasiequivseeds}
		Two seeds $\mathbf{s} = \bigl((A_j)_{j\in J}, J_{\mathrm{fr}},B,D\bigr)$ and $\mathbf{s}' = \bigl((A'_j)_{j\in J'}, J'_{\mathrm{fr}},B',D'\bigr)$ with cluster variables in the same field $K$ are \emph{quasi-equivalent} if we have:
		\begin{enumerate}[(1)]
			\item A bijection $\phi:J\to J'$ such that $\phi(J_{\mathrm{fr}})=J'_{\mathrm{fr}}$, and $d_{\phi(j)}=d'_{j}$.
			\item The frozen variables of $\mathbf{s}$ and $\mathbf{s}'$ generate the same multiplicative subgroup of $K$.
			\item For $j\in J_{\mathrm{uf}}$, the variables $A_{j}$ and $A_{\phi(j)}'$ differ by a Laurent monomial of frozens.
			\item For $j\in J_{\mathrm{uf}}$, the \emph{exchange ratios} are preserved. That is, $y_{j} = y_{\phi(j)}'$ where
			$$y_j = \frac{\prod_{\ep_{ij}>0}A_i^{\ep_{ij}}}{\prod_{\ep_{ij}<0}A_i^{-\ep_{ij}}} = \prod_{i}A_i^{\ep_{ij}}.$$

		\end{enumerate}
	\end{defn}
	\begin{rem}
		Preservation of the exchange ratios is the necessary condition for quasi-equivalences to commute with mutation. This can be equivalently thought of as the $2$-forms agreeing on the respective tori, see \citep[Definition 3.8]{GLSB25}.
	\end{rem}
	
	\begin{lem}[\protect{\citep[Proposition 2.3]{Fra16}}]
		If seeds $\mathbf{s}$ and $\mathbf{s}'$ are quasi-equivalent, then for any mutable index $j$, $\mu_{j}(s)$ and $\mu_{\phi(j)}(\mathbf{s}')$ are quasi equivalent.
	\end{lem}

	\begin{defn}[\citep{Fra16}]
		A \emph{quasi-cluster morphism} is a homomorphism $ \psi: \cA(\seed) \to \cA(\seed')$ between cluster algebras of the same rank (i.e. $|J_{\mathrm{uf}}|=|J'_{\mathrm{uf}}|$) such that for every $j\in J_{\mathrm{uf}}$, $\psi(A_j)/A'_j$ is a Laurent monomial in frozen variables, the image of a frozen variable is a monomial of frozen variables, and exchange ratios are preserved. That is, quasi-equivalent seeds arise from quasi-cluster morphisms with a quasi-cluster inverse.
	\end{defn}
	
	\begin{defn}[\protect{\citep[Definition 4.1]{BY25}}]
		Suppose we have an embedding of fields $\sigma:K\to K'$ and cluster algebras $\cA\subset K$, $\cA'\subset K'$ with initial seeds $\mathbf{s}$, $\mathbf{s}'$ respectively.     
		If we have a subset $J_{\mathrm{sup}}\subseteq J$ containing the mutable indices $J_{\mathrm{uf}} \subseteq J_{\mathrm{sup}}$, an injection $\iota: J_{\mathrm{sup}}\hookrightarrow J'$ such that $\iota(J_{\mathrm{uf}}) \subseteq J_{\mathrm{uf}}'$, and a subset $S'\subseteq J'\setminus \iota(J_{\mathrm{sup}})$ satisfy the following three properties:
		\begin{enumerate}[(1)]
			\item for any $i\in J_{\mathrm{sup}}$, $\sigma(A_i)$ and $A'_{\iota(i)}$ differ by a Laurent monomial in $A'_j$, for $j\in S'$,
			\item for any $i\in J_{\mathrm{uf}}$, the exchange ratios satisfy $\sigma(X_i) = X'_{\iota(i)}$, where
			$$X_i = \prod_{j\in J}A_j^{\ep_{ij}} \quad \text{and} \quad X'_{\iota(i)} = \prod_{j\in J'}(A'_j)^{\ep'_{\iota(i)j}},$$
			\item for any $i\in J\setminus J_{\mathrm{sup}}$, $\sigma(A_i)$ is a Laurent monomial in the $A'_j$, for $j\in S'$,
		\end{enumerate}
		then $\sigma: K\hookrightarrow K'$ is called \emph{rational quasi-cluster} and is denoted by $\begin{tikzcd}
			\sigma: (\cA,\mathbf{s}) \ar[r,dashed,"(J_{\mathrm{sup}} { , }S')"] & (\cA',\mathbf{s}')
		\end{tikzcd}$.
	\end{defn}
	When $J_{\mathrm{sup}} = J_{\mathrm{uf}}$ and $S'=J'_{\mathrm{fr}}$, and $|J_{\mathrm{uf}}|=|J_{\mathrm{uf}}'|$, this is a quasi-cluster morphism.
	The advantage of rational quasi-cluster maps is that neither ranks (number of mutable variables) nor dimensions are necessarily preserved.

		\subsection{Compatible Poisson brackets}
	
	A \emph{Poisson bracket} $\{\cdot,\cdot\}:\mathcal{A} \times \mathcal{A} \to \mathcal{A}$ on a $\C$-algebra $\mathcal{A}$ is a Lie bracket which is a $\C$-derivation in both entries.
	Poisson brackets readily extend to the field of fractions $\mathrm{Frac}(\mathcal{A})$.
	A \emph{Poisson structure} on an algebraic variety $X$ will be a Poisson bracket on its ring of regular functions, or equivalently its field of rational functions. This is equivalent to the data of a \emph{Poisson bivector field} $\pi\in \Gamma(X,\largewedge^2TX)$ satisfying the Schouten bracket condition $[\pi,\pi]=0$.
	Elements $f_1,\dots,f_n\in \mathcal{A}$ are called \emph{log-canonical}  if $\{f_i,f_j\} = \omega_{ij}f_if_j$, for some constants $\omega_{ij}\in \Z$.
	
	\begin{defn}[\citep{GSV03}]
		Given a cluster algebra $\mathcal{A} = \mathcal{A}(\seed_0)$, a Poisson structure on the ambient field $\mathcal{F}$ is \emph{compatible} if and only if all extended clusters are log-canonical.
	\end{defn}
	
	\begin{lem}[\protect{\citep[Theorem 1.4, Equation (1.5)]{GSV03}, see also \citep[Definition 3.1]{BZ05}}]\label{lem:poissoncompatibilitysufficent}\leavevmode \\
		Let $\Omega = (\omega_{ij})$ be the Poisson coefficient matrix of a log-canonical initial cluster with a (tall) extended exchange matrix $\widetilde{E} = E [I_{{\mathrm{uf}}} \; 0]^T$ given by the mutable columns of $E$. Then the Poisson structure is compatible if and only if $\widetilde{E}^T \Omega = [\Delta \;\; 0]$, for $\Delta$ a diagonal matrix.
	\end{lem}
	While a cluster $\mathcal{X}$-variety possesses a natural cluster-Poisson structure, this might in principle be a different Poisson structure to other known Poisson structures. See Remark 1.11 in \citep{SW21}. The vector space of Poisson structures compatible with a given cluster structure is described in \citep[Theorem 1.4]{GSV03}.

	\section{Bott--Samelson (sub-)varieties and charts}\label{sec:BSvarieties}

	\subsection{Notation and conventions}
	
	Let $G$ be a simple algebraic group. Fix a pair of opposite Borel subgroups $B$ and $B^-$, their unipotent radicals $U$ and $U^-$, and maximal torus $T=B\cap B^-$. The simple roots $\alpha_i$ are indexed by a set $I$, and the Weyl group $W=N_G(T)/T$ is generated by simple reflections $s_i$ for $i\in I$. We fix a \emph{pinning}, so for each $i\in I$, we have isomorphisms $u_i:\mathbb{A}^1\to U_{\alpha_i}$ and $u_{-i}:\mathbb{A}^1\to U_{-\alpha_i}$ so that
	\begin{equation*}
		\begin{pmatrix}
			1 & z \\0 & 1
		\end{pmatrix} \mapsto u_i(z), \qquad \begin{pmatrix}
			x & 0 \\0 & x^{-1}
		\end{pmatrix} \mapsto \alpha_i^\vee(x), \qquad \begin{pmatrix}
			1 & 0 \\z & 1
		\end{pmatrix} \mapsto u_{-i}(z),
	\end{equation*}
	defines a homomorphism $\varphi_i:SL_2\to G$. Our Cartan matrix will be denoted $A=(a_{ij})_{i,j\in I}$ such that $\alpha_i(\alpha_j^\vee(x)) = x^{a_{ij}}$ for any $x\in \mathbb{G}_m$. 	The standard lift of the simple reflection is
	$$\dot{s}_i = u_{i}(-1)u_{-i}(1)u_i(-1) = \varphi_i\begin{pmatrix}
		0 & -1\\1 &0
	\end{pmatrix} \in N_G(T).$$
	For an element $w\in W$ of the Weyl group with reduced expression $w=s_{i_1}s_{i_2}\cdots s_{i_\ell}$, its standard lift $\dot{w} = \dot{s}_{i_1}\dot{s}_{i_2}\cdots \dot{s}_{i_\ell}$ is known to be independent of choice of reduced expression.
	
	A \emph{subexpression} of a word $\underline{w}$ amounts to choosing some positions to omit.
	In \citep{EL21}, omitted letters in a word of simple reflections were replaced with instances of the identity, so $(s_1,e,s_1)$ was a subexpression of $(s_1,s_2,s_1)$. In \citep{E16}, these positions were marked with blanks and this example would be $(s_1,\blank,s_1)$.
	For what follows, it will be necessary to track what specific letters which were removed, and a subexpression of $\underline{w} = (i_1,\dots,i_\ell) \in I^\ell$ will be denoted by an element of $ \gamma \in \Upsilon_{\underline{w}} = \{i_1, -i_1\} \times \cdots \times \{i_\ell, -i_\ell\} \subseteq (\pm I)^\ell.$
	Positive letters in a subexpression will be called \emph{supported}, and negative letters will be \emph{unsupported}.
	For $i\in \pm I$, we write $s_i^+ = s_i$ if $i>0$ and $s_i^+ = e$ if $i<0$.
	Then for a subexpression $\gamma$, we have the sequence $s_\gamma = (s_{i_1}^+,s_{i_2}^+, \dots, s_{i_\ell}^+)$ which corresponds to the notation used in \citep{EL21}.
	For $1\leq k\leq \ell$, we let $\gamma^k=s^+_{\gamma_1}s^+_{\gamma_2}\cdots s^+_{\gamma_k}$ denote the element in $W$ given by the $k$-th partial product of the subexpression.
	\begin{rem}
		These signed expressions should not be confused with the various other instances of words in two copies of the same alphabet (e.g. in \citep{SW21} or \citep{GLSB25}).
	\end{rem}

	If $\underline{x} = (x_1,\dots,x_\ell)$ is an $\ell$-tuple of invertible elements in a ring and $c=(c_1,\dots,c_\ell)^T$ is a column vector of integers, we write $\underline{x}^c = x_1^{c_1}\cdots x_\ell^{c_\ell}$. If $C=(c_{ij})$ is an $\ell\times m$ integer matrix with columns $C_1,\dots,C_m$, we write $\underline{x}^C = (\underline{x}^{C_1},\dots \underline{x}^{C_m})$.

	\subsection{Bott--Samelsons, bricks, cells, and braids}
	We introduce various subvarieties of the twisted product of flags
	$G\times^B G \times^B  \cdots \times^B G/B$
	including Bott--Samelson varieties, brick manifolds, braid varieties.
	Flags in $G/B$ will be denoted by calligraphic letters $\mathcal{B} \in G/B$, or by coset representatives $gB\in G/B$. Two flags $xB, yB$ are in \emph{relative position} $w\in W$ if and only if $x^{-1}yB \in BwB$, and this will be denoted by $xB \xrightarrow{w}yB$.	
	For $i\in I$, let $P_i$ denote the parabolic subgroup $P_i = B \sqcup Bs_iB \subset G.$ For a word $\underline{w}\in I^\ell$, consider the product variety $P_{\underline{w}} := P_{i_1}\times P_{i_2} \times \cdots \times P_{i_\ell} \subset G^\ell$ and endow it with a $B^\ell$-action given by 
	\begin{equation}\label{eqn:BSaction}
		(g_1,g_2,\dots,g_\ell) * (b_1,\dots b_\ell) = (g_1b_1,b_1^{-1}g_2b_2,\dots,b_{l-1}^{-1}g_\ell b_\ell).
	\end{equation}
	The \emph{Bott--Samelson variety} associated to $\underline{w}$ is the quotient space
	$Z_{\underline{w}} := P_{\underline{w}}/B^\ell$.
	It is endowed with the structure of a projective variety via the $B^\ell$-equivariant embedding
	$$Z_{\underline{w}} \hookrightarrow (G/B)^\ell,\quad [g_1,g_2,\dots,g_\ell] \mapsto \bigl(g_1B, g_1g_2 B, \dots, g_1{\cdots}\,g_\ell B\bigr),$$
	so we can equivalently view the Bott--Samelson variety as comprised of sequences of flags,
	$$Z_{\underline{w}} = \left\{(\mathcal{B}_0,\mathcal{B}_1,\dots,\mathcal{B}_\ell) \in (G/B)^{\ell+1} \;\middle|\; \mathcal{B}_0 = B \text{, and }   \mathcal{B}_{j-1} = \mathcal{B}_{j} \text{ or } \mathcal{B}_{j-1} \xrightarrow{s_{i_j}} \mathcal{B}_{j} \right\}.$$
	Note that each flag $\mathcal{B}_j$ is chosen out of a $\P^1$ of possible flags which differ from $\mathcal{B}_{j-1}$ in the $i_j$ component. This gives the Bott--Samelson variety the structure of an iterated $\P^1$-fibration.
	
	The multiplication map $m: Z_{\underline{w}} \to G/B$ is given by $[g_1,g_2,\dots,g_\ell] \mapsto g_1g_2\cdots g_\ell B$. For sequences of flags this is projection onto the last entry.
	The \emph{brick manifold} \citep{E16} is the closed subvariety $m^{-1}(\delta(\underline{w}) B)\subset Z_{\underline{w}}$, consisting of those sequences whose final flag differs from the initial flag by the Demazure product $\delta(\underline{w})$. When $\delta(\underline{w})=w_0$ is the longest element, the brick manifold is equal to $m^{-1}(B^-w_0B/B) \subset Z_{\underline{w}}$.
	
	The \emph{Bott--Samelson cell} is the open subvariety of $\mathring{Z}_{\underline{w}}\subset {Z_{\underline{w}}}$ of consecutively distinct flags:
	\begin{align*}
		\mathring{Z}_{\underline{w}} &= (Bs_{i_1}B) \times^B \cdots \times^B (Bs_{i_\ell}B)/B,\\ &= \left\{(\mathcal{B}_0,\mathcal{B}_1,\dots,\mathcal{B}_\ell) \in (G/B)^{\ell+1} \;\middle|\; \mathcal{B}_0 = B, \mathcal{B}_{j-1} \xrightarrow{s_{i_j}} \mathcal{B}_{j} \right\}.
	\end{align*}
	The unique factorisation $Bs_iB = U_{\alpha_i} \dot s_i B$, where $U_{\alpha_i}$ is the $1$-parameter subgroup associated to the simple root $\alpha$, can be used to iteratively trivialise the $\proj^1$-fibration from the left to right.
	This gives an isomorphism to an affine space of dimension $\ell$. When $\underline{w}$ is a reduced word, the cell $\mathring{Z}_{\underline{w}}$ is isomorphic to the Schubert cell $\mathring{X_w} = BwB/B$ where $w= \delta(\underline{w})$.
	
	The intersection of the brick manifold with the Bott--Samelson cell $\mathring{Z}_{\underline{w}} \cap m^{-1}(\delta(\underline{w})B)$ is called the \emph{braid variety} $X(\underline{w})$. In \citep{CGGS24,CGGLSS25}, it is shown that applying braid moves to the word
	gives isomorphic braid varieties.
	Furthermore, letters which strictly increase the Demazure product can be added to the word without changing the isomorphism class.

	For $v\in W$ such that $v\leq \delta(\underline{w})$, we define the \emph{open Richardson variety} $\mathring{Z}_{\underline{w}}^{v} = \mathring{Z}_{\underline{w}}\cap m^{-1}(B^-vB/B)$ inside the twisted product of flags. When $\underline{w}$ is a reduced word, the isomorphism of the Bott--Samelson cell to the Schubert cell $BwB/B$ gives an isomorphism of the $\mathring{Z}_{\underline{w}}^{v}$ with the usual open Richardson varieties $\mathring{R}_w^v\subseteq G/B$ in the flag variety. The open Richardson $\mathring{Z}_{\underline{w}}^{v}$ is isomorphic to the braid variety for a concatenated word $X(\underline{w},\underline{v^{-1}w_0})$, where $\underline{v^{-1}w_0}$ denotes is a choice of reduced word $v^{-1}w_0\in W$.
	
	When $\delta(\underline{w})= w_0$, we have the braid variety as an open Richardson variety
	$$X(\underline{w}) = \mathring{Z}_{\underline{w}}^{w_0} =  \left\{(\mathcal{B}_0,\mathcal{B}_1,\dots,\mathcal{B}_\ell) \in (G/B)^{\ell+1} \;\middle|\; \mathcal{B}_0 = B, \mathcal{B}_{j-1} \xrightarrow{s_{i_j}} \mathcal{B}_{j},  \mathcal{B}_\ell = w_0B  \right\}.$$

	\begin{rem}\label{rem:braidvarisopenpartofBScell}
		The braid variety is defined as a closed variety inside the Bott--Samelson cell with the same word. There is also an isomorphism $X((\underline{w_0},\underline{w})) \cong \mathring{Z}_{\underline{w}}\cap m^{-1}({B^-B/B}) = \mathring{Z}_{\underline{w}}^e$ between an open subset of a Bott--Samelson cell and a braid variety of a longer word \citep[Lemma 3.16]{CGGLSS25}.
		The complement of this open set is the union of $\mathring{Z}_{\underline{w}}^{s_i} = \mathring{Z}_{\underline{w}}\cap m^{-1}({B^-s_iB/B})$ where the $s_i$ range over all simple reflections, or equivalently the non-vanishing of the generalised minors \citep[\S1.4]{FZ99} 
		evaluated on the final flag $\cB_\ell$.
	\end{rem}

	\subsection{The standard seed for the Bott--Samelson cell}\label{sec:almagmatedseeds}
	A particular initial seed for the Bott--Samelson cell $\mathring{Z}_{\underline{w}} = \cO^{\underline{w}}$ is ubiquitous in the literature. The quivers appear in \citep{GLS11} as endomorphism algebras of preprojective algebra modules in their additive categorification of Schubert cells. It comes from the \emph{right inductive weave} \citep[Definition 4.5]{CGGLSS25} for the braid variety $X(\underline{w_0},\underline{w})$, which agrees by \citep[\S4.8]{CGGLSS25} with the initial seed from \citep{SW21} constructed by amalgamation. The cluster torus is the open strata from the Deodhar decomposition of the flag variety, which gives the first seed in \citep{GLSB25}.
	In \citep{LM26}, it is shown that this is also the standard initial seed arising from the procedure in \citep{GY18} for the symmetric Poisson CGL presentation on $\cO^{\underline{w}}$.
	In type $A_n$, when $\underline{w}$ is a reduced word of the longest Grassmannian element, this is the \emph{rectangles seed} considered in \citep{RW19,RW25,FWZ25} comprised of the Pl\"ucker coordinates whose Young tableaux presentations are given by rectangles. This last example is interesting since Grassmannian elements are \emph{fully commutative} \citep{Stembridge96}, so all choices of reduced words yield this cluster.
	
	A detailed construction of the quiver by \emph{amalgamation} is given in \citep{SW21} and \citep{BY25}.
	The complete exchange matrix can be described by an interlacing condition \citep{BFZ05, GLS11}. For  $\underline{w} =(i_1,\dots,i_\ell)$, define the \emph{predecessor} and \emph{successor} functions
	\begin{equation}\label{eqn:pred}
		p(j) = \begin{cases}
			\max\{k<j \mid i_k=i_j\} &\text{if } \exists k<j \text{ with } i_k=i_j,\\
			-\infty &\text{else},
		\end{cases}
	\end{equation}
	and
	\begin{equation}\label{eqn:succ}
		s(j) = \begin{cases}
			\min\{k>j \mid i_k=i_j\} &\text{if } \exists k>j \text{ with } i_k=i_j,\\
			\infty &\text{else}.
		\end{cases}
	\end{equation}
	Then the entries of the exchange matrix $E(\underline{w}) = (\ep_{jk})_{1\leq j,k \leq \ell}$ are given by 
	\begin{equation}\label{eqn:exchangematrixofW}
		\ep_{jk} :=\begin{cases}
			a_{i_k,i_j} & \text{if } k<j<s(k)< s(j),\\
			-a_{i_k,i_j} & \text{if } j<k<s(j)< s(k),\\
			\frac{a_{i_k,i_j}}{2} & \text{if } k<j<s(k)= s(j)=\infty,\\
			-\frac{a_{i_k,i_j}}{2} & \text{if } j<k<s(j)= s(k)=\infty,\\
			1 & \text{if } j=s(k),\\				
			-1 & \text{if } j=p(k),\\				
			0 & \text{else.}
		\end{cases}
	\end{equation}
	
	It will also be convenient to denote the vertices of this seed by tuples $v_j=(i_j,m)$, which encodes the $m$-th appearance of $i_j$ in $\underline{w}$.
	The cluster variables of this seed $A_j$ are given by generalised minors of the partial products $\varphi_j = \Delta_{\omega_{i_j},\omega_{i_j}}\circ m\circ \psi_{\underline{w},j}$.
	Specifically, for the $j$-th letter of $\underline{w}$ and its vertex $v_j=(i_j,m)$, we have the cluster variable $A_j = A_{(i_j,m)} = \varphi_j :=  \Delta_{\omega_{i_j},\omega_{i_j}}(B_{\underline{w}[j]}(z_1,\dots,z_j))$ for the Bott--Samelson coordinates $z_i$ on $\cO^{\underline{w}}$. The frozen variables are those given by the last instance of each letter in $\underline{w}$ and are  $\Delta_{\omega_{i},\omega_{i}}(B_{\underline{w}}(z_1,\dots,z_\ell))$ for $i\in \mathrm{supp}(\underline{w})$.
	Every polynomial generator $z_j$ is reachable from this seed by a sequence of mutations \citep{GLS11} (see also \citep{SW21,CGGLSS25,GLSB25} or Definition \ref{defn:levelmutation}) and therefore $$\cA(\seed(\underline{w})) = \C[\mathring{Z}_{\underline{w}} \cap m^{-1}(B^-B/B)]\quad \text{and} \quad \cA(\seed(\underline{w}),\emptyset) = \C[\mathring{Z}_{\underline{w}}].$$
	
	\begin{rem}
		Similar exchange matrices are constructed for expressions in two alphabets of simple roots \citep{SW21}. Letters $\overline{i}$ for $i\in I$ are associated to a component matrix $E(\overline{i})$ which is the negation of $E(i)$, and these are amalgamated similarly. We reiterate that signed letters in subexpressions $\gamma$ have a distinct meaning which we introduce in the next subsection.
	\end{rem}

	\subsection{Bott--Samelson charts}
	The Bott--Samelson variety is an iterated $\proj^1$-fibration: at the $j$-th stage we have a new entry in $P_{i_j}/B \cong \proj^1$. Equivalently, under the embedding $Z_{\underline{w}} \hookrightarrow (G/B)^\ell$, at each stage, we modify the $i_j$-th component of the flag out of a $\proj^1$ of possible subspaces to get the next flag in the sequence. With this fibration structure, the variety has a number of natural affine charts, of which the Bott--Samelson cell $\mathring{Z}_{\underline{w}}$ is one.
	
	For the Bott--Samelson cell, we imposed conditions $g_jB \neq B$, that consecutive flags are not equal. To obtain other charts, we instead use the opposite affine chart in $\proj^1$ in certain positions, and ask that $g_jB \neq s_{i_j}B$. We index these by the subexpressions $\gamma \in \Upsilon_{\underline{w}}$ of $\underline{w}$. In terms of relative positions, when $\gamma_j = -s_{i_j}$, we have that $\mathcal{B}_{j-1} = \mathcal{B}_j$ or $\mathcal{B}_{j-1} \xrightarrow{s_{i_j}} \mathcal{B}_j$, but $\mathcal{B}_{j-1} \neq s_{i_j}\cdot \mathcal{B}_j$. We can uniformly describe these charts as the images of affine spaces under particular embeddings dependent on the subexpression $\gamma= (\gamma_1,\gamma_2,\dots,\gamma_\ell)\in \Upsilon_{\underline{w}}$. For a subexpression $\gamma\in \Upsilon_{\underline{w}}$, the \emph{Bott--Samelson chart} $\cO^{\gamma}$ is the image of the embedding
	$$\Phi_{\gamma}: \mathbb{A}^\ell \to Z_{\underline{w}}, \quad (z_1,\dots,z_l) \mapsto \left[u_{\gamma_1}(z_1) \dot{s}^+_{\gamma_1},u_{\gamma_2}(z_2) \dot{s}^+_{\gamma_2},\dots,u_{\gamma_\ell}(z_\ell) \dot{s}^+_{\gamma_\ell}\right].$$ Note that the chart for the ``full'' subexpression $\gamma=\underline{w}$ is the Bott--Samelson cell.
	
	We call two charts $\cO^\gamma$ and $\cO^{\gamma'}$ \emph{adjacent}, if there is a unique index $k$ such that $\gamma_k=-\gamma_k'$. One should picture the $1$-skeleton of an $\ell$-dimensional hypercube with vertices labelled by subexpressions and edges connecting adjacent charts.
	\begin{rem}
		The change of coordinates between any two Bott--Samelson charts can be computed as the composition of a sequence of coordinate changes between adjacent charts.
	\end{rem}

	\begin{defn}\label{defn:braidmatrix}
		For a subexpression $\gamma \in \Upsilon_{\underline{w}}$ and $\underline{z}=(z_1,\dots,z_\ell)\in \mathbb{A}^\ell$, we define the \emph{braid matrix} to be the element of $G$ given by the product $$B_\gamma(\underline{z})=B_\gamma(z_1,\dots,z_\ell) := u_{\gamma_1}(z_1) \dot{s}^+_{\gamma_1}u_{\gamma_2}(z_2) \dot{s}^+_{\gamma_2}\cdots u_{\gamma_\ell}(z_\ell) \dot{s}^+_{\gamma_\ell}.$$
		This construction is such that $m\circ\Phi_{\gamma}(\underline{z})  = B_\gamma(\underline{z}) \cdot B\in G/B$.
	\end{defn}
	
	Similar to Remark \ref{rem:braidvarisopenpartofBScell}, the fibre of the multiplication map over the longest element can be used to study an open subset of $\cO^\gamma$. The following is essentially \citep[Lemma 3.16]{CGGLSS25} adapted for an arbitrary Bott--Samelson chart.
	
	\begin{lem}\label{lem:braidopeninchart}
		Let $\gamma$ be a signed expression, and $\underline{w_0}=(i_1,\dots,i_r)\in I^r$ a reduced expression of the longest element with positive letters. There is a natural isomorphism
		$$\cO^{(\underline{w_0},\gamma)} \cap m^{-1}(w_0B/B) \cong \cO^\gamma \cap m^{-1}(B^-B/B).$$
	\end{lem}
	\begin{proof}
		Let $\zeta_1,\dots\zeta_N,z_1,\dots z_\ell$ denote the coordinates on the Bott--Samelson chart $\cO^{(\underline{w_0},\gamma)}$, where the $\zeta_i$ correspond to the letters in $\underline{w_0}$, and $z_i$ those in $\gamma$. By construction,
		$$m\circ\Phi_{(\underline{w_0},\gamma)}(\underline{\zeta},\underline{z}) = \dot{w_0}B \in G/B \;\text{ if and only if }\;
		\dot{w_0}\cdot B_{\underline{w_0}}(\underline{\zeta}) B_\gamma(\underline{z}) \in B.$$
		Note that the map $\underline{\zeta} \mapsto w_0\cdot B_{\underline{w_0}}(\underline{\zeta}) $ gives an isomorphism $\mathbb{A}^N\to w_0Uw_0 =U^-$. Hence, $B_\gamma(\underline{z}) \in B^-B$ and $m\circ\Phi_\gamma(\underline{z})\in B^-B/B$. Conversely, for $\Phi_\gamma(\underline{z}) \in\cO^\gamma \cap m^{-1}(B^-B/B)$, the product $B_\gamma(\underline{z})\in U^-B$ factors uniquely as $B_\gamma(\underline{z})=x_-x_+$ with $x_-\in U^-,x_+\in B$. For the unique $\underline{\zeta}$ such that
		$\dot{w_0}\cdot B_{\underline{w_0}}(\underline{\zeta})= (x_-)^{-1}$, we obtain $\dot{w_0}\cdot B_{\underline{w_0}}(\underline{\zeta}) B_\gamma(\underline{z}) \in B$ as required.
	\end{proof}

	\subsection{Morphisms between Bott--Samelson varieties and their charts}
	For any subset $J = \{j_1<\dots<j_m\}\subset [\ell]=\left\{1,\dots,\ell\right\}$, let $\underline{w}[J] := \left(i_{j_1},\dots,i_{j_m}\right)$.
	There is a closed embedding $\iota_{w[J]}: Z_{\underline{w}[J]} \to Z_{\underline{w}}$ defined by $[g_{1},\dots, g_{m}] \mapsto [e,\dots, e, g_{1}, e,\dots,e, g_{m},e,\dots,e],$ 
	where the output has $g_l$ in position $j_l$ for $1\leq l\leq m$ and $e$ in all positions in $[\ell]\setminus J$.
	
	For $1\leq m \leq \ell$, we have the truncated word of length $m$ given by $\underline{w}[m] = (i_1,\dots,i_m)$. The truncation morphism $\psi_{\underline{w},m}: Z_{\underline{w}} \to Z_{\underline{w}[m]}$, $[g_1,\dots, g_\ell]\to [g_1,\dots, g_m]$ is given by projecting to the first $m$ entries or equivalently by forgetting the last $\ell-m$ flags. For $m=\ell-1$, we abbreviate to $\psi_{\underline{w},\ell-1}$ to $\psi_{\underline{w}}$ so that
	$\psi_{\underline{w},m} = \psi_{\underline{w}[m+1]} \circ \cdots \circ \psi_{\underline{w}[\ell-1]} \circ \psi_{\underline{w}}$.
	Each $\psi_{\underline{w},m}$ is a locally trivial $Z_{\underline{w}[\ell-m+1,\ell]}$-fibration, with a section $\sigma_{\underline{w},m} = \iota_{\underline{w}[m]}$.
	These truncations restrict to Bott--Samelson charts $\psi_{\underline{w}}:\cO^\gamma \to \cO^{\gamma[\ell-1]}$, but the embeddings $\iota_{\underline{w}[J]}$ only restrict to $\cO^\gamma$ if the target chart is unsupported at all positions in $[\ell]\setminus J$.

	\begin{lem}\label{lem:truncationmapsarequasicluster}
		The truncation morphism $\psi_{\underline{w},m}:\cO^{\underline{w}} \to \cO^{\underline{w}[m]}$ is a rational quasi-cluster.
	\end{lem}
	\begin{proof}
		This arises from the embedding of seeds $\seed(\underline{w}[m]) \hookrightarrow \seed(\underline{w})$. The cluster variables are the same generalised minors of the same partial products, and the exchange ratios are preserved as the mutable parts of the quiver for $\seed(\underline{w}[m])$ agrees with its image in $\seed(\underline{w})$.
	\end{proof}

	\subsection{The standard Poisson structure}
	A \emph{Poisson algebraic group} $(G,\pi_G)$ is a an algebraic group $G$ with a Poisson structure $\pi_G$ such that the group multiplication  $$m: (G\times G, \pi_{G} \times \pi_{G}) \to (G,\pi_{G})$$ is a Poisson map.
	Every connected complex semi-simple algebraic group can be endowed with its \emph{standard Poisson structure} $\pi_{\mathrm{st}}$, which is defined using a choice of opposite Borel subgroups.
	For a fixed symmetric non-degenerate invariant bilinear form $\langle\cdot,\cdot\rangle$ on $\mathfrak{g}$ (and its induced form on $\mathfrak{h}^*$), the Poisson bivector for the standard Poisson structure is
	\begin{equation*}
		\pi_{\mathrm{st}} = \sum_{\alpha\in \Phi_+} \frac{\langle\alpha,\alpha\rangle}{2} e_{-\alpha}\wedge e_{\alpha} \in \largewedge^2\mathfrak{g}.
	\end{equation*}
	Note that this bivector depends on the choice of non-degenerate bilinear form $\langle\cdot,\cdot\rangle$, but not on the choice of pinning. A canonical choice of this form is such that $\langle\alpha,\alpha\rangle =2$ for the short roots so that $\langle\alpha,\alpha\rangle \in \{2,4,6\}$ for all simple roots, so $\langle\alpha_i, \alpha_j\rangle = a_{ij}d_j$. The standard Poisson structure is the semi-classical limit of the associated quantum group.
	
	The standard Poisson structure on the Bott--Samelson variety  \citep{EL21} is defined by by taking the product Poisson structure $(\pi_{\mathrm{st}})^\ell$ on product of Poisson algebraic subgroups $P_{\underline{w}}$ and pushing it forward onto the quotient space $Z_{\underline{w}}$. Elek and Lu then calculated the Poisson bracket in the Bott--Samelson coordinates on each chart $\cO^\gamma$.
	
	\begin{lem}[\protect{\citep[Lemma 3.2, Theorem 4.14]{EL21}}]\label{lem:StdPoissonStr}
		In the Bott--Samelson coordinates, the Poisson structure on $\cO^{\gamma} \subset Z_{\underline{w}}$ is given by		
		\begin{equation}\label{eqn:StdPoissonStr}
			\{z_j,z_k\} = \begin{cases}
				\langle \gamma^j(\alpha_{i_j}), \gamma^k (\alpha_{i_k}) \rangle z_j z_k & \text{ if } \gamma_j < 0, \\
				-\langle \gamma^j(\alpha_{i_j}), \gamma^k (\alpha_{i_k}) \rangle z_j z_k  - \langle\alpha_{i_j},\alpha_{i_j}\rangle \sigma_j(z_k) & \text{ if } \gamma_j >0,
			\end{cases}			
		\end{equation}
		for $j<k$. Here, $\sigma_j(z_k)$ denotes the action of the vector field $\sigma_j(p) = \frac{d}{dt}|_{t=0}\bigl(\exp(te_{\alpha_{i_j}})\cdot p\bigr)$ on $z_k$ as a local function on $ \cO^{(\gamma_{j+1},\dots,\gamma_\ell)} \subset Z_{({i_{j+1}},\dots, {i_\ell})}$ via the parametrisation
		$$\mathbb{A}^{\ell-j}\ni (z_{j+1},\dots,z_\ell) \mapsto \Phi_{(\gamma_{j+1},\dots,\gamma_\ell)}(z_{j+1},\dots,z_\ell).
		$$
	\end{lem}
	\begin{rem}
		The formula (\ref{eqn:StdPoissonStr}) is calculated by considering the action of the bivector of the $j$-th factor on the rest of the variety. The first term in the formula arises from a torus component, and the second arises from a unipotent component. A recursive formula for $\sigma_j(z_k)$ is given in \citep[Theorem 4.10]{EL21}. In particular, $\sigma_j(z_k)$ is polynomial in the coordinates strictly between positions $j$ and $k$, and at most quadratic in $z_k$.
	\end{rem}
	\begin{rem}
		When $\gamma_j<0$, its Bott--Samelson coordinate $z_j$ is log-canonical with every Bott--Samelson coordinate to its right, but need not be so with the coordinates to its left. If we could somehow ``move'' this negative letter and its coordinate to a position to the left of all the others, it could then be log-canonical with every other coordinate.
	\end{rem}
	\begin{cor}
		The isomorphism in Lemma \ref{lem:braidopeninchart} is a Poisson isomorphism.
	\end{cor}
	\begin{proof}
		For $\cO^{(\underline{w_0},\gamma)} \cap m^{-1}(w_0B)$, we observed that the coordinates $\zeta_i$ for positions in $\underline{w_0}$ are determined by the coordinates $z_i$ for positions in $\gamma$. It therefore suffices to check that the bracket on the $z_i$ agree on both sides. This now follows from Lemma \ref{lem:StdPoissonStr} since the bracket $\{z_j,z_k\}$  depends only the signed expression in between positions $j$ and $k$.
	\end{proof}
	
	\begin{eg}\label{eg:(1,2,1,-2)}
		In type $A$, let $\underline{w}=(1,2,1,2)$ and $\gamma = (1,2,1,-2)$. Denote coordinates for the Bott--Samelson charts 
		$\cO^{\underline{w}}$ and $\cO^{\gamma}$ by $(z_1,\dots,z_4)$ and $(z'_1,\dots,z'_4)$ respectively.
		Calculating explicitly using Lemma \ref{lem:StdPoissonStr}, we obtain
		\begin{equation*}
			\begin{array}{l l l}
				\{z_1,z_2\}_{\underline{w}} = -z_1z_2, & \{z_1,z_3\}_{\underline{w}} = z_1z_3-2z_2, & \{z_1,z_4\}_{\underline{w}} = 2z_1z_4-2,\\
				& \{z_2,z_3\}_{\underline{w}}=-z_2z_3,& \{z_2,z_4\}_{\underline{w}} = z_2z_4-2z_3,\\
				&& \{z_3,z_4\}_{\underline{w}}= -z_3z_4.
			\end{array}
		\end{equation*}
		and
		\begin{equation*}
			\begin{array}{l l l}
				\{z'_1,z'_2\}_{\gamma} = -z'_1z'_2, & \{z'_1,z'_3\}_{\gamma} = z'_1z'_3-2z'_2, & \{z'_1,z'_4\}_{\gamma} = -2z'_1z'_4+(z'_4)^2,\\
				& \{z'_2,z'_3\}_{\gamma}=-z'_2z'_3,& \{z'_2,z'_4\}_{\gamma} = -z'_2z'_4 + 2z'_3(z'_4)^2,\\
				&& \{z'_3,z'_4\}_{\gamma}= z'_3z'_4.
			\end{array}
		\end{equation*}
	\end{eg}

	\begin{lem}
		For any subexpression $\gamma$, the intersection $\cO^{-\gamma}\cap \cO^{\gamma}$ of two \emph{opposite} Bott--Samelson charts in $Z_{\underline{w}}$ is isomorphic to a torus $(\mathbb{G}_m)^\ell$.
	\end{lem}
	\begin{proof}
		We proceed by induction on the length $\ell$. For $\ell=1$, $Z_{\underline{w}}$ is isomorphic to $\P^1$, and the intersection $\cO^{(-i_i)}\cap \cO^{({i_1})} \cong \P^1 \setminus\{0,\infty\}=\mathbb{G}_m$.
		Let $\gamma'=\gamma[\ell-1] = (\gamma_1,\dots,\gamma_{\ell-1})$ be the truncation of $\gamma$, and suppose that $\cO^{\gamma'} \cap \cO^{-{{\gamma'}}} \cong \mathbb{G}_m^{\ell-1}$. The truncation map $\psi_{\underline{w}}:Z_{\underline{w}} \to Z_{\underline{w}[\ell-1]}$ is a $\P^1$-fibration which can be trivialised over any of the Bott--Samelson charts. Restricting the domain to the chart $\cO^{\gamma}$, we have a trivial $\mathbb{A}^1$-fibration $\psi_{\underline{w}}: \cO^\gamma \to \cO^{\gamma'}$ for the two cases $\gamma = (\gamma',-i_\ell)$ and $\gamma = (\gamma', i_\ell)$. Suppose we have $p\in Z_{\underline{w}} \setminus \cO^\gamma$ such that $p'=\psi_{\underline{w}}(p) \in \cO^{\gamma'} \cap \cO^{{-\gamma'}}$.
		In the first case, when $\gamma_\ell={i_\ell}$, we have $$p=\Phi_{(\gamma',-i_\ell)}(z_1,\dots,z_{\ell-1},0) = \bigl[\Phi_{\gamma'}(z_1,\dots,z_{\ell-1}), e\bigr].$$
		Since $p'\in \cO^{\gamma'} \cap \cO^{-{\gamma'}}$, we have 
		$p = \Phi_{(-\gamma',-i_\ell)}\bigl(\zeta_1(z_1),\dots,\zeta_{\ell-1}(z_1,\dots,z_{\ell-1}),0\bigr) \in \cO^{-{\gamma}}$
		where the $\zeta_i$ denote the change of coordinates between $\cO^{\gamma'}$ and $\cO^{-{\gamma'}}$.
		
		In the other case, when $\gamma_\ell=-i_\ell$, and our condition $p\notin \cO^\gamma$ implies we must have $$p=\Phi_{(\gamma',i_\ell)}(z_1,\dots,z_{\ell-1},0) = \bigl[\Phi_{\gamma'}(z_1,\dots,z_{\ell-1}), \dot{s}_{i_\ell}\bigr].$$
		Since, $p'\in \cO^{\gamma'} \cap \cO^{-{\gamma'}}$, we can rewrite $p'$ in coordinates for $\cO^{-{\gamma'}}$, to get
		$$p = \Phi_{-{\gamma}}\left(\zeta_1(z_1),\zeta_2(z_1,z_2),\dots,\zeta_{\ell-1}(z_1,\dots,z_{\ell-1}), \zeta_\ell(z_1,\dots,z_{\ell-1})\right) \in \cO^{-{\gamma}}.$$	
		Thus, the charts $\cO^{\gamma}$ and $\cO^{-\gamma}$ cover the trivial $\P^1$-fibration over $\cO^{\gamma'} \cap \cO^{-\gamma'} \cong \mathbb{G}_m^{\ell-1}$. Their complements are non-intersecting sections, so intersect as a trivial $\mathbb{G}_m$-fibration over $\mathbb{G}_m^{\ell-1}$, and $\cO^{\gamma}\cap \cO^{-{\gamma}} \cong \mathbb{G}_m^{\ell}$.
	\end{proof}
	
	Observe that (\ref{eqn:StdPoissonStr}) says the Bott--Samelson coordinates for the ``unsupported chart'' $\cO^{-\underline{w}} = \cO^{(-i_1,,-i_2,\dots,-i_\ell)}$ are log-canonical, and $\{z_j\neq0\} \subset \cO^{-\underline{w}}$ is a log-canonical torus.
	
	\begin{cor}\label{cor:emptytorusislogcanonical}
		The torus $\cO^{-\underline{w}}\cap \cO^{\underline{w}}$ is given by the complement of the coordinate hyperplanes in $\cO^{-\underline{w}}$, and is hence a log-canonical torus in $\cO^{\underline{w}}$.
	\end{cor}
	
	\begin{rem}
		This is the torus corresponding to the seed $\seed(\underline{w})$, see Lemma \ref{lem:monomialchangematrix}.
		This gives a quick verification that the initial clusters of \citep{SW21,CGGLSS25,GLSB25} in $\cO^{\underline{w}}$ are log-canonical with the standard Poisson structure.
		Furthermore, in $\cO^{-\underline{w}}$, since this torus is the complement of the coordinate hyperplanes, if we take each $\zeta_i$ to be frozen, then we trivially have $\cA = \mathcal{U} = \C[\cO^{-\underline{w}}]$, for an initial seed comprised solely of noninvertible frozens.
	\end{rem}

	We can determine the complementary hypersurfaces for this torus in $\cO^{\underline{w}}$ by finding the change of coordinates from $\cO^{\underline{w}}$ to $\cO^{-\underline{w}}$. To do this, we iteratively invert the leftmost coordinate and modify the other functions accordingly. This rational map factors as a composition of $\ell$ changes of coordinates through adjacent Bott--Samelson charts
	$$\cO^{\underline{w}} = \cO^{(i_1,i_2,\dots,i_\ell)} \dashrightarrow \cO^{(-i_1,i_2,\dots,i_\ell)} \dashrightarrow \cO^{(-i_1,-i_2,\dots,i_\ell)}\dashrightarrow \cdots \dashrightarrow \cO^{-\underline{w}}.$$
	This idea of finding rational maps which are defined on a log-canonical torus is the basis of our interpretation of Demazure weaves, and we expanded upon in later sections.

	\section{Demazure weaves and the standard Poisson structure}\label{sec:classical weaves}
	\subsection{Demazure weaves}\label{subsec:DemazureWeaves}	
	A \emph{Demazure weave} \citep[\S4.1]{CGGLSS25} is a planar graph consisting of $I$-coloured edges and vertices of specific valences arising from relations between braid matrices. Figure \ref{fig:classicalweavevertices} below is recreated from \citep{CGGLSS25} for reference.

	\begin{figure}[htbp]
		\centering
		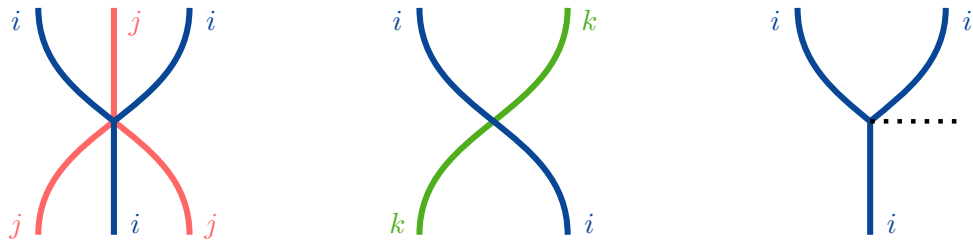
		\caption{The three types of vertices in classical Demazure weaves of simply-laced type. We keep to the convention in \citep{CGGLSS25} of using {\textbf{\textcolor{S1}{blue}}}, {\textbf{\textcolor{S2}{red}}}, and {\textbf{\textcolor{S3}{green}}} to denote edges labelled $i,j,k$ where $i$ and $j$ are adjacent roots, and $i$ and $k$ are not adjacent.}
		\label{fig:classicalweavevertices}
	\end{figure}

	Each (generic) horizontal slice of a Demazure weave is a word with letters in $I$. In \citep{CGGLSS25}, the hexavalent and tetravalent vertices correspond to biregular isomorphisms between cells arising from applying braid moves to these words. Demazure weaves diagrammatically capture the data of a sequence of rational maps between Bott--Samelson cells and the trivalent vertices correspond to rational maps defined on the non-vanishing of a particular rational function.
	Each strand carries a rational function in the initial Bott--Samelson coordinates $z_1,\dots,z_\ell$. The dotted ray propagating rightward from the trivalent vertex is called a \emph{raking ray} (see \citep[Definition 5.1]{CGGLSS25}), and is used to modify each label to the right of a trivalent vertex.
	The respective matrix identities for the three vertices in Figure \ref{fig:classicalweavevertices} are
	\begin{align}
		B_i(z_1)B_j(z_2)B_i(z_3) &= B_j(z_3)B_i(z_1z_3-z_2)B_j(z_1),\\  B_i(z_1)B_k(z_2) &= B_k(z_2)B_i(z_1), \text{ and}\\ B_i(z_1)B_i(z_2) &= B_i(z_1-z_2^{-1})\alpha^\vee_i(z_2)u_{i}(-z_2^{-1}). \label{eqn:trivalentidentity}
	\end{align}
	In (\ref{eqn:trivalentidentity}), the element $\alpha^\vee_i(z_2)u_{i}(-z_2^{-1})$ of the Borel subgroup labels the raking ray. It propagates rightwards and modifies the other coordinates appropriately.	
	\begin{rem}
		The conventions for weaves in \citep{CGGS24} and \citep{CGGLSS25} differ by a horizontal flip of the diagrams and transposition of the matrices. We follow the convention in \citep{CGGLSS25}, which matches the construction of Bott--Samelson charts in \citep{EL21}.
	\end{rem}
	
	\begin{rem}
		In non-simply-laced type, higher valency vertices are needed. These correspond to relations between roots of a $B_2$ or $G_2$ sub-root system, see \citep[\S6.1]{CGGLSS25}.
	\end{rem}
	
	It is shown that over a whole weave, each trivalent vertex only adds one new irreducible component to the indeterminacy locus of the overall composed map. The other zeroes or poles were previously acquired on previous trivalent vertices, and so these can be collected to obtain a toric chart on the braid variety (or on Bott--Samelson cells). Weave equivalences are defined between weaves which give the same toric chart. It is shown \citep[Lemma 4.4]{CGGLSS25} that any two weaves are related by a sequence of weave equivalences and weave mutations.
	
	Weaves can be interpreted as giving conditions for a moduli of flags. Given a weave $\mathfrak{w}$ we can label each region in the plane with a flag $\cB$ such that whenever two flags are separated by an $i$-coloured strand, they are in relative position $s_i$. If $g\cdot B$ is the flag to the left of an $i$-coloured strand with variable $z$, then the flag on the right of the strand is $gB_i(z)\cdot B$. In this sense, the flags along top-most slice of a Demazure weave correspond to the open condition of the Bott--Samelson cell. After the variables at the bottom-most strands are declared to be zero, the flags along the bottom-most slice of a weave contribute the closed condition for the braid variety: that the left-most and right-most flags differ by exactly $\delta(\underline{w})$. The subset $T_{\mathfrak{w}}$ of the braid variety where the interior regions can be filled by flags will be a cluster torus.
	
	\begin{rem}
		For open Richardson varieties $\mathring{Z}^v_{\underline{w}} = \mathring{Z}_{\underline{w}}\cap m^{-1}(B^{-}vB)$, we would need to impose the condition that the last flag is in the $B^-$ orbit of $vB$, but this construction of weaves only allows us to impose a condition of being in the fibre $m^{-1}(vB/B)$ of torus fixed points. These coincide for the longest element $w_0$, but not for other elements $v$. This partly inspired the ``negative strands'' of the later sections.
	\end{rem}

	\subsection{Weave tori are log-canonical in the standard Poisson structure}\label{subsec:Poisson weaves}
	Equipping the Bott--Samelson cells with the standard Poisson structure, we have the following results.
	\begin{prop}\label{prop:Poissonmapsbetweenshortcells}
		The maps between $2$- and $3$-dimensional Bott--Samelson cells  corresponding to tetravalent and hexavalent vertices are Poisson isomorphisms.
	\end{prop}
	\begin{proof}		
		For non-adjacent roots $i$ and $k$, the map $\cO^{\gamma}=\cO^{(i,k)}\to \cO^{(k,i)}= \cO^{\widetilde{\gamma}}$ is given by $ \Phi_{\gamma}(z_1,z_2)\mapsto \Phi_{\widetilde{\gamma}}(z_2,z_1)$. By Lemma \ref{lem:StdPoissonStr}, the bracket between the coordinate functions is trivial on both the domain and codomain of this isomorphism, and so the map is Poisson.
		
		For adjacent roots $i$ and $j$, let $\gamma={(i,j,i)}$ and $\widetilde{\gamma}=(j,i,j)$. Consider $f:\cO^{\gamma} \to \cO^{\widetilde{\gamma}}$ defined by $\Phi_{\gamma}(z_1,z_2,z_3) \mapsto \Phi_{\widetilde{\gamma}}(z_3,z_1z_3-z_2,z_1)$.
		Using Lemma \ref{lem:StdPoissonStr}, the Poisson bracket on $\cO^{\gamma}$ in the Bott--Samelson coordinates is determined by
		$$\{z_1,z_2\}_{\gamma} = -z_1z_2,\quad \{z_1,z_3\}_{\gamma} =z_1z_3-2z_2,\quad \{z_2,z_3\}_{\gamma}=-z_2z_3,$$
		and similarly for the coordinates $\widetilde{z}_i$ on $\cO^{\widetilde{\gamma}}$.
		We conclude that $f:\cO^\gamma \to \cO^{\widetilde{\gamma}}$ is Poisson since
		\begin{align*}
			f^*\left(\{\widetilde{z}_1,\widetilde{z}_2\}_{\widetilde{\gamma}}\right)  &= 
			f^*\left(-\widetilde{z}_1\widetilde{z}_2\right)=
			-z_3(z_1z_3-z_2)\\&= \{z_3,z_1z_3-z_2\}_{\gamma} = \{f^*(\widetilde{z}_1),f^*(\widetilde{z}_2)\}_{\gamma} \,,\\
			f^*\left(\{\widetilde{z}_1,\widetilde{z}_3\}_{\widetilde{\gamma}}\right)  &=
			f^*\left(\widetilde{z}_1\widetilde{z}_3-2\widetilde{z}_2\right) = 
			-(z_1z_3-2z_2)\\& = \{z_3,z_1\}_{\gamma} = \{f^*(\widetilde{z}_1),f^*(\widetilde{z}_3)\}_{\gamma} \,, \; \text{ and}\\
			f^*\left(\{\widetilde{z}_2,\widetilde{z}_3\}_{\widetilde{\gamma}}\right) & =
			f^*\left(-\widetilde{z}_2\widetilde{z}_3\right)=
			-z_1(z_1z_3-z_2) \\&= \{z_1z_3-z_2,z_1\}_{\gamma} = \{f^*(\widetilde{z}_2),f^*(\widetilde{z}_3)\}_{\gamma}\,. \qedhere
		\end{align*}
	\end{proof}
	
	\begin{rem}
		The above isomorphisms are also a result of these Bott--Samelson cells being Poisson isomorphic to Schubert cells under their respective multiplication maps.
	\end{rem}
	Weave vertices therefore give Poisson isomorphisms between $2$- and $3$-dimensional Bott--Samelson cells. We now show that these isomorphisms can be used within a subexpression.
	\begin{lem}\label{lem:Poissonfactoring}
		Let $X,Y,Z$ be Poisson varieties, $\phi:X\to Y$ a surjective Poisson map, and $\psi:Y\to Z$ a map such that the composition $\psi\circ\phi:X\to Y $ is Poisson. Then $\psi$ is Poisson.
	\end{lem}
	\begin{proof}
		For rational functions $f,g$ on $Z$, we have
		$$\{f,g\}_Z\circ (\psi\circ\phi) = \{f\circ (\psi\circ\phi),g\circ (\psi\circ\phi)\}_X = \{f\circ \psi,g\circ \psi\}_Y \circ \phi.$$
		We conclude that $\{f,g\}_Z\circ \psi = \{f\circ \psi,g\circ \psi\}_Y$ as $\phi$ is surjective.
	\end{proof}

	\begin{prop}\label{prop:Poissonmapsbetweencells}
		Let $m:\cO^{\gamma}\to G/B$ and $\widetilde{m}:\cO^{\widetilde{\gamma}} \to G/B$ denote the multiplication maps from two Bott--Samelson charts. A Poisson map $f: \cO^{\gamma} \to \cO^{\widetilde{\gamma}}$ such that $\widetilde{m}\circ f = m$
		induces a Poisson map $\widetilde{f}: \cO^{(\gamma',\gamma,\gamma'')} \to \cO^{(\gamma',\widetilde{\gamma},\gamma'')}$ between higher dimensional charts.
		
	\end{prop}
	\begin{proof}
		For $\gamma=(\gamma_1,\dots,\gamma_\ell)$, let $\cO^\gamma\cdot B$ denote the right $B$-orbit of the image of the map $\Phi_\gamma$ in $P_{\gamma_1}\times^B\cdots \times^BP_{\gamma_\ell}$. Note the lack of a quotient by $B$ on the rightmost factor.
		
		The Bott--Samelson charts can be trivialised by their intermediary products. That is, $\cO^{(\gamma',\gamma,\gamma'')}$ is Poisson isomorphic to $(\cO^{\gamma'}\cdot B)\times^B (\cO^{\gamma}\cdot B) \times^B \cO^{\gamma''}$. Suppose we have a Poisson regular map $f:\cO^{\gamma} \to \cO^{\widetilde{\gamma}}$ between Bott--Samelson charts. Then we have a natural map between the product spaces given by taking the identity on the other two factors. That the map $f$ respects the multiplication maps implies that $\mathrm{id}\times f\times \mathrm{id}$ is $B^2$-equivariant and hence descends to a map between the quotient spaces. We have the following diagram:
		\[\begin{tikzcd}
			{(\cO^{\gamma'}\cdot B)\times (\cO^{\gamma}\cdot B) \times \cO^{\gamma''}} && {(\cO^{\gamma'}\cdot B)\times (\cO^{\widetilde{\gamma}}\cdot B) \times \cO^{\gamma''}} \\
			{\cO^{\gamma'} \times^B \cO^{\gamma} \times^B \cO^{\gamma''}} && {\cO^{\gamma'} \times^B \cO^{\widetilde{\gamma}} \times^B \cO^{\gamma''}}
			\arrow["{\mathrm{id}\times f\times \mathrm{id}}", from=1-1, to=1-3]
			\arrow["q"', from=1-1, to=2-1]
			\arrow["{\widetilde{q}}", from=1-3, to=2-3]
			\arrow["{\widetilde{f}}"', from=2-1, to=2-3]
		\end{tikzcd}\]
		Note that $\widetilde{q} \circ (\mathrm{id}\times f\times \mathrm{id})$ is a Poisson map as a composition of two Poisson maps, and that $q$ is a surjective Poisson map. Applying Lemma \ref{lem:Poissonfactoring}, we conclude that the induced map $\widetilde{f}$ between the higher dimensional Bott--Samelson charts is Poisson.
	\end{proof}

	\begin{lem}\label{lem:Poissoncontraction}
		Let $\underline{w}$ be a word with $i_{j-1}= i_j$ and denote by $\widetilde{\underline{w}}$ the word of length $\ell-1$ obtained by omitting the $j$-th letter. There is a canonical Poisson map $m_j: Z_{\underline{w}} \to Z_{\widetilde{\underline{w}}}$ obtained by combining the $(j-1)$ and $j$-th entries, 
		$$ [g_1,\dots,g_{j-2},g_{j-1},g_j,g_{j+1},\dots, g_\ell] \mapsto [g_1,\dots,g_{j-2},g_{j-1}g_j, g_{j+1},\dots, g_\ell].$$
	\end{lem}
	\begin{proof}
		Let $\underline{w}=(\underline{w'},i,i,\underline{w''})$. The parabolic subgroup $P_i$ is a Poisson algebraic subgroup of $G$, so applying the multiplication map to the two $P_i$ factors in $ P_{\underline{w'}}\times P_i\times P_i\times P_{\underline{w''}} \to P_{\underline{w'}}\times P_i\times P_{\underline{w''}}$ defines a $B^\ell$-equivariant map which descends to a map on Bott--Samelson varieties $m_j: Z_{\underline{w}} \to Z_{\widetilde{\underline{w}}}$. The map is Poisson by Lemma \ref{lem:Poissonfactoring}.
	\end{proof}
	
	\begin{rem}
		If we consider the Bott--Samelson varieties using their embeddings as sequences of flags, this map corresponds to forgetting the $(j-1)$-th flag in the sequence.
	\end{rem}
	
	To show that trivalent vertices respect the standard Poisson structure, it is helpful to factor the rational map
	between Bott--Samelson cells through an adjacent Bott--Samelson chart:
	\[\begin{tikzcd}[row sep=tiny]
		& {\cO^{(\gamma',i,-i,\gamma'')}} & \\
		{\cO^{(\gamma',i,i,\gamma'')}} && {\cO^{(\gamma',i,\gamma'')}}
		\arrow[from=1-2, to=2-3]
		\arrow[dashed, from=2-1, to=1-2]
		\arrow[from=2-1, to=2-3]
	\end{tikzcd}\]
	The first factor is a transition map within the same Poisson variety $Z_{\underline{w}}$ and is therefore  a Poisson map.
	There are rational functions  $z_{j+1}',\cdots, z_\ell'$  in the $z_i$ such that the map is
	$$\Phi_{(w',i,i,w'')}(z_1,\dots,z_{j-1},z_j,z_{j+1},\dots, z_\ell) \mapsto \Phi_{(w',i,-i,w'')}(z_1,\dots,z_{j-1},z_j^{-1},z_{j+1}',\dots, z_\ell').$$
	
	\begin{lem}\label{lem:frayedtwiningisPoisson}
		Let $\underline{w}$ be a word with $i_{j-1}= i_j = i$, and let $\gamma = (\gamma_1, \dots, \gamma_{j-1},\gamma_j, \dots,\gamma_\ell)$ be a subexpression with $\gamma_{j-1}=i_{j-1}$ and $\gamma_j=-i_j=-i_{j-1}$.		
		Let $\gamma^+$ be obtained by omitting the $j$-th entry of $\gamma$.
		Then we have a biregular isomorphism $f: \cO^{\gamma} \to \cO^{\gamma^+} \times \mathbb{A}^1$
		given by
		$$\Phi_{\gamma}(z_1,\dots,z_{j-1},z_j,\dots,z_n) \mapsto \bigl(\Phi_{\gamma^+}(z_1,\dots,z_{j-2},z_{j-1}-z_j,z_{j+1},\dots,z_\ell), \; z_j\bigr).$$
		Furthermore, if $\gamma_k>0$ implies $\gamma_{k'}>0$ for all $k<k'<j$, then we can endow the codomain with a log-canonical bracket on the product space to make this isomorphism a Poisson map.
		
	\end{lem}
	\begin{proof}
		The map onto the first factor is a result of writing the Poisson map in Lemma \ref{lem:Poissoncontraction} in the Bott--Samelson coordinates for these charts. Alternatively, a case-by-case argument could check the bracket in Lemma \ref{lem:StdPoissonStr} on each of the coordinate functions.
		
		The log-canonicity of the coordinate $z_j$ in the $\mathbb{A}^1$ factor with the other coordinate functions arises from the Elek--Lu formula (\ref{eqn:StdPoissonStr}). In the domain, $z_j$ is log-canonical with $z_k$ for $k>j$ since the $j$-th letter is unsupported. For the log-canonicity of $z_{j-1}-z_j$ with $z_j$, the Elek--Lu formula gives that $\{z_{j-1},z_j\} = -\langle\alpha_i,\alpha_i\rangle z_{j-1}z_j+\langle\alpha_i,\alpha_i\rangle z_j^2$, and therefore $$\{z_{j-1}-z_j,z_j\} = -\langle\alpha_i,\alpha_i\rangle(z_{j-1}-z_j)z_j.$$
		
		Finally, for $k<(j-1)$, if $\gamma_k$ is negative, then log-canonicity of $z_k$ is immediate from (\ref{eqn:StdPoissonStr}). If $\gamma_k$ is positive, then our assumption on the positivity of the letters between $k$ and $j$ means that by \citep[Lemma 4.6]{EL21} the vector field acting on $z_{j-1}$ has no torus component. We therefore conclude $\sigma_{\alpha_{i_j}}(z_j) = 0$ by \citep[Lemma 4.6 (1)]{EL21}.
	\end{proof}

	\begin{thm}\label{thm:weavechartsarelogcan}
		Each toric chart $T_{\mathfrak{w}} \subset X(\underline{w})$ obtained from a Demazure weave $\mathfrak{w}:(\underline{w})\to \delta(\underline{w})$ is log-canonical with respect to the standard Poisson structure.
	\end{thm}
	\begin{proof}
		The preceding discussion verifies that the map in \citep[Lemma 2.22]{CGGS24} is Poisson. Together with Proposition \ref{prop:Poissonmapsbetweenshortcells}, each component map of a Demazure weave is Poisson. The torus associated to a weave is obtained by iteratively factoring out a torus factor from trivalent vertices. Our positivity assumption on the subexpression $\gamma$ in Lemma \ref{lem:frayedtwiningisPoisson} applies to Bott--Samelson cells, and therefore we have the log-canonicity of this torus factor the other coordinates at each step. Proposition \ref{prop:Poissonmapsbetweenshortcells} gives that the isomorphisms coming from tetravalent and hexavalent vertices are indeed Poisson isomorphisms, and therefore the log-canonicity is preserved at each stage between the trivalent vertices.
	\end{proof}

	\subsection{Compatibility of the cluster and Poisson structures}
	
	The previous subsection showed that every weave gives a log-canonical torus. However, for quivers of infinite mutation type, there can be seeds which are not represented by any weave. In this subsection, we use Lemma \ref{lem:poissoncompatibilitysufficent} and check that all such cluster tori are log-canonical. 
	\begin{rem}
		The statements in this subsection for the Bott--Samelson cell are known since the Bott--Samelson coordinates give a symmetric Poisson CGL presentation on the coordinate ring of $\cO^{\underline{w}}$ \citep[Theorem 5.12]{EL21}. The coincidence of the initial seed $\seed(\underline{w})$ with the Poisson CGL initial seed from \citep{GY18} is checked in \citep{LM26}. For completeness, we provide a self-contained check of compatibility without using Poisson CGL theory.
	\end{rem}
	
	Recall the successor and predecessor functions from (\ref{eqn:pred}) and  (\ref{eqn:succ}). Define two $\ell\times \ell$ upper triangular integer matrices $\Lambda= \Lambda(\underline{w}) = (\lambda_{jk})$ and $Q=Q(\underline{w})=(q_{jk})$ with entries
	\begin{equation*}
		\lambda_{jk} = \begin{cases}
			1 &\text{if } k=j,\\
			-1 &\text{if } k=s(j),\\
			0 & \text{else},
		\end{cases} \quad \text{and} \quad
		q_{jk} = \begin{cases}
			-1 &\text{if } k=j,\\
			-a_{i_k,i_j} &\text{if } k>j,\\
			0 & \text{else}.
		\end{cases}
	\end{equation*}
	
	\begin{lem}\label{lem:monomialchangematrix}
		Let $\zeta_i$ be the Bott--Samelson coordinates on $\cO^{-\underline{w}}$ and recall the initial cluster variables $\varphi_j$ on $\cO^{\underline{w}}$. On the intersection $\cO^{-\underline{w}} \cap \cO^{-\underline{w}}$, we have $\underline{\zeta} = \underline{\varphi}^{\Lambda Q}$.
	\end{lem}
	\begin{proof}
		Let $C=\Lambda Q = (c_{jk})$.
		The claim is that for each $1\leq j\leq \ell$, we have $\zeta_k = \varphi_1^{c_{1k}} \varphi_2^{c_{2k}} \cdots \varphi_k^{c_{kk}}$ where 
		\begin{equation}\label{eqn:monomialchangematrix}
			c_{jk} = \begin{cases}
				-1 &\text{if } j=k,\\
				-a_{i_k,i_j} &\text{if } j<k<s(j),\\
				-2 - (-1) &\text{if } s(j)=k,\\
				-a_{i_k,i_j} - (-a_{i_k,i_{s(j)}}) &\text{if } s(j)<k,\\
				0 & \text{else},
			\end{cases}
			= \begin{cases}
				-1 &\text{if } j=k,\\
				-a_{i_k,i_j} &\text{if } j<k<s(j),\\
				-1 &\text{if } k=s(j),\\
				0 & \text{else}.
			\end{cases}
		\end{equation}
		We use induction on $k$.
		When $k=1$, the identity $B_{i}(z) = u_{-i}(z^{-1}) \alpha_i^\vee(z) u_{i}(-z^{-1})$ for the root $i=i_k$ gives that $\zeta=\frac{1}{\varphi_{1}}$ as required.
		When $p(k) = -\infty$ we have,
		\begin{equation*}
			\varphi_{k} = 
			\Delta_{\omega_{i_k},\omega_{i_k}} \left(B_{\underline{w}[k-1]}(z_1,\dots,z_{k-1}) \cdot B_{i_k}(z_k)\right)= \Delta_{\omega_{i_k},\omega_{i_k}} \left(B_{i_k}(z_k)\right)=z_k.
		\end{equation*}
		For $k>1$ with $p(k) \neq -\infty$ and an appropriate unipotent matrix $U$, we have
		\begin{align}
			\varphi_{k} &= \Delta_{\omega_{i_k},\omega_{i_k}} (B_{\underline{w}[k]}(z_1,\dots,z_k)) \nonumber \\
			&= \Delta_{\omega_{i_k},\omega_{i_k}} \left(B_{\underline{w}[k-1]}(z_1,\dots,z_{k-1}) \cdot B_{i_k}(z_k)\right) \nonumber\\
			&= \Delta_{\omega_{i_k},\omega_{i_k}} \left(\alpha_{i_k}^\vee(\varphi_{p(k)}) B_{i_k}(z_k')\right)\nonumber\\
			&= \Delta_{\omega_{i_k},\omega_{i_k}} \left( \alpha_{i_k}^\vee(\varphi_{p(k)}) u_{-i_k}((z_k')^{-1})\alpha_{i_k}(z_k') u_{i_k}(-(z_k')^{-1})\right)\nonumber\\
			& = \varphi_{p(k)}z_k' \label{eqn:successiveminors}
		\end{align}
		where $z_k' = z_k'(z_1,\dots,z_k)$ is the polynomial resulting from moving the unipotent matrix $U$ rightward. Applying the coordinate change and commuting matrices, we have
		\begin{align*}
			B_{\underline{w}[k]}(z_1,\dots,z_k) &= B_{\underline{w}[k-1]}(z_1,\dots,z_{k-1}) \cdot B_{i_k}(z_k),\\
			&= \left(\prod_{j=1}^{k-1}u_{-i_j}(\zeta_j)\right)u_{-i_k}\left((z_k' )^{-1}\left(\prod_{j<k \,:\; s(j)\geq k}\varphi_j^{-a_{i_k,i_j}} \right)\right) b
		\end{align*}
		for an appropriate Borel element $b\in B$.
		From this, we apply (\ref{eqn:successiveminors}) and conclude
		\begin{align*}
			\zeta_k &= (z_k' )^{-1}\left(\prod_{j<k \,:\; s(j)\geq k}\varphi_j^{-a_{i_k,i_j}}\right)
			= \varphi_{s(k)}^{-1} \varphi_k^{-1} \left(\prod_{j<k \,:\; s(j)> k}\varphi_j^{-a_{i_k,i_j}}\right).\qedhere
		\end{align*}
	\end{proof}                        
	
	The initial cluster variables $\varphi_j$ of the Bott--Samelson cell $\cO^{\underline{w}}$ and the coordinates $\zeta_j$ of the opposite chart $\cO^{-\underline{w}}$ are related by a monomial transformation. The Elek--Lu formula (\ref{eqn:StdPoissonStr}) gives that the $\zeta_j$ are log-canonical, and therefore the $\varphi_j$ are log-canonical.
	
	\begin{prop}\label{prop:Poissonconjugation}
		Let $\underline{\zeta}=(\zeta_1,\dots,\zeta_\ell)$ and  $\underline{\varphi}=(\varphi_{1},\dots,\varphi_\ell)$ be two collections of log-canonical coordinates with Poisson coefficient matrices $\varTheta=(\vartheta_{ij})$ and $\Omega=(\omega_{ij})$ so that
		$$\{\zeta_i,\zeta_j\}=\vartheta_{ij}\zeta_i\zeta_j \quad \text{and} \quad \{\varphi_i,\varphi_j\}=\omega_{ij}\varphi_i\varphi_j.$$		
		If $\underline{\zeta}=\underline{\varphi}^C$ for an integer matrix $C=(c_{ij})$, then we have $\varTheta=C^T\Omega C$.
	\end{prop}
	\begin{proof}
		With the monomial transformation $\underline{\zeta}=\underline{\varphi}^C$ and expanding with the Leibniz rule,
		\begin{equation*}
			\{\zeta_i,\zeta_j\} =  \left\{\prod_{k=1}^{\ell}\varphi_k^{c_{ki}},\zeta_j\right\}
			=  \sum_{k=1}^\ell c_{ki}\frac{\varphi_1^{c_{1i}}\cdots\varphi_{\ell}^{c_{\ell i}}}{\varphi_k} \{\varphi_k, \zeta_j\}
			=  \left(\sum_{k,k'=1}^\ell c_{ki}\omega_{kk'}c_{k'j}\right) \zeta_i\zeta_j.
		\end{equation*}
		Therefore, $\vartheta_{ij} = \sum_{k'=1}^\ell\left(\left(\sum_{k=1}^\ell c_{ki}\omega_{kk'}\right)c_{k'j}\right)$, and $\varTheta=C^T\Omega C$.
	\end{proof}
	
	From (\ref{eqn:StdPoissonStr}), we know that the Poisson coefficient matrix of $(\zeta_1,\dots, \zeta_\ell)$ has entries $\langle \alpha_{i_j},\alpha_{i_k}\rangle$ for $j<k$, and hence $\varTheta=(\varDelta Q)^T-\varDelta Q$ where $\varDelta = \mathrm{diag}(\delta_1,\dots,\delta_\ell)$  is a diagonal matrix whose entries
	$\delta_j = \mathrm{lcm}(D) \cdot d_{i_j}^{-1}$
	are from the symmetriser $D=\mathrm{diag}(d_1,\dots,d_n)$ of the Cartan matrix. Therefore, the Poisson coefficient matrix $\Omega$ for the seed $\seed(\underline{w})$ is 
	\begin{equation}\label{eqn:PoissonCoefficientMatrix}
		\begin{split}
			\Omega = (C^T)^{-1}\varTheta C^{-1} &= \left((\Lambda Q)^T\right)^{-1}\left((\varDelta Q)^T-\varDelta Q\right)(\Lambda Q)^{-1},\\
			&= (\Lambda^T)^{-1}\left( \varDelta Q^{-1}-(Q^T)^{-1}\varDelta \right) \Lambda^{-1}.
		\end{split}
	\end{equation}
	
	\begin{lem}\label{lem:RectangularExchangeMatrix}
		The matrices $E(\underline{w})$ and $\Lambda Q \Lambda^T$ agree on the mutable columns.
	\end{lem}
	\begin{proof}
		Again, let $C=\Lambda Q$. Whenever $s(k)<\infty$, we have $(C \Lambda^T)_{j,k} = c_{j,k} - c_{j,s(k)}$.
		We can then check case-by-case that using the description of $C$ from (\ref{eqn:monomialchangematrix}), that whenever $s(k)<\infty$,
		\begin{equation*}
			(\Lambda Q \Lambda^T)_{j,k} = (C \Lambda^T)_{j,k}=\begin{cases}
				a_{i_k,i_j} & \text{if } k<j<s(k)\leq s(j),\\
				-a_{i_k,i_j} & \text{if } j<k<s(j)\leq s(k),\\
				1 & \text{if } j=s(k),\\				
				-1 & \text{if } j=p(k),\\				
				0 & \text{else},
			\end{cases}
		\end{equation*}
		which matches the entries in exchange matrix from (\ref{eqn:exchangematrixofW}).
	\end{proof}
	
	\begin{lem}\label{lem:auxillarymatrixidentity}
		We have  $\left[I_{\mathrm{uf}} \;\; 0\right](\varDelta ^{-1}\Lambda Q^T \varDelta ) = -\left[I_{\mathrm{uf}} \;\; 0\right](\Lambda Q)$.
	\end{lem}
	\begin{proof}
		We need to show that the two matrices $\varDelta ^{-1}\Lambda Q^T \varDelta $ and $-\Lambda Q$ agree on their mutable rows.
		Note that when $s(j)<\infty$, we have $(\Lambda Q^T)_{j,k} = q_{k,j} - q_{k,s(j)}$. Hence,
		\begin{equation*}
			(\Lambda Q^T)_{j,k} = \begin{cases}
				1 & \text{if } j=k,\\
				a_{i_j,i_k} & \text{if } j<k<s(j),\\
				1 & \text{if } j=s(k),\\
				0 & \text{else.}
			\end{cases}
		\end{equation*}
		We have $(\varDelta^{-1}\Lambda Q^T\varDelta)_{j,k} = \frac{\delta_k}{\delta_j} \cdot (\Lambda Q^T)_{j,k} $.	
		Since $\varDelta$ was defined in terms of the symmetrisers of the Cartan matrix, $a_{i_j,i_k} \delta_k = a_{i_k,i_j} \delta_j$, and we are done after comparing with (\ref{eqn:monomialchangematrix}).
	\end{proof}
	
	\begin{prop}\label{prop:cellclusteralgebraisPoissoncompatible}
		The cluster algebra structure $\cA(\seed(\underline{w}))$ on $\cO^{\underline{w}}= \mathring{Z}_{\underline{w}}$ is compatible with the standard Poisson structure.
	\end{prop}
	\begin{proof}
		Using Lemma \ref{lem:monomialchangematrix}, we can determine the Poisson coefficient matrix for the cluster variables $A_j = \varphi_j$ in the seed $\seed(\underline{w})$. We then utilise the criterion in Lemma \ref{lem:poissoncompatibilitysufficent}.
		
		Applying Lemma \ref{lem:RectangularExchangeMatrix}, and (\ref{eqn:PoissonCoefficientMatrix}), up to reordering indices, we can simplify
		\begin{align*}
			\widetilde{E}^T \Omega &= \left(\left[I_{\mathrm{uf}} \;\; 0\right] \Lambda Q^T \Lambda^T\right) \left((\Lambda^T)^{-1}\left( \varDelta Q^{-1}-(Q^T)^{-1}\varDelta \right) \Lambda^{-1}\right),\\
			&=\left[I_{\mathrm{uf}} \;\; 0\right]\varDelta  \cdot  \left(\left(\varDelta ^{-1}\Lambda Q^T\varDelta \right)\left(Q^{-1}\Lambda^{-1}\right) - I_\ell\right).
		\end{align*}
		Then applying Lemma \ref{lem:auxillarymatrixidentity}, $\widetilde{E}^T \Omega 
		=\left[I_{\mathrm{uf}} \;\; 0\right]\varDelta   \left(-(\Lambda Q) Q^{-1}\Lambda^{-1} - I_\ell\right)=-2\left[I_{\mathrm{uf}} \;\; 0\right]\varDelta$
		is comprised of a diagonal block and a zero block. Thus, the exchange matrix $\widetilde{E}$ and the Poisson coefficient matrix $\Omega$ form a compatible pair \citep{BZ05} and Lemma \ref{lem:poissoncompatibilitysufficent} says the cluster and Poisson structures are compatible.
	\end{proof}

	\begin{cor}\label{cor:ClusterPoissonCompatible}
		The cluster structure constructed on braid varieties in \citep{CGGLSS25} and \citep{GLSB25} is compatible with the standard Poisson structure.
	\end{cor}
	\begin{proof}
		Since the left inductive weave of the braid variety $X(\underline{w_0},\underline{w})$  the seed for the Bott--Samelson cell, we have the statement in the case that $w_0$ is a prefix by Corollary \ref{prop:cellclusteralgebraisPoissoncompatible}.
		As shown in the proofs of \citep[Theorem 5.22]{CGGLSS25} and its component lemmas, reachable seeds of $X(\underline{w})$ are obtained by iterated freezings of particular seeds from $X(\underline{w_0},\underline{w})$, together with deletions by specialising frozen variables, so log-canonicity of the clusters is inherited.
	\end{proof}
	
	\pagebreak
	\section{Weaves for non-full Bott--Samelson charts}\label{sec:weavesfornonfullcharts}
	\subsection{New vertices with frayed strands}
	In Section \ref{subsec:Poisson weaves}, checking that trivalent vertices are Poisson was made substantially easier by first factoring the rational map through an adjacent (and not-fully-supported) Bott--Samelson chart. This motivates a construction of weaves for an arbitrary Bott--Samelson chart.
	
	In a classical Demazure weave, we used $I$-coloured strands to separate flags in relative position $s_i$. For Bott--Samelson charts, negative letters indicate when two consecutive flags might be equal or in relative position $s_i$, but not opposite. For this, we will introduce new strands coloured by labels in $-I$. These are informally referred to as \emph{frayed strands}. In diagrams, these are denoted with dashed lines, which are indicate the regions they separate might correspond to equal flags when the corresponding affine parameter is zero.
	
	\begin{figure}[htbp]
		\centering
		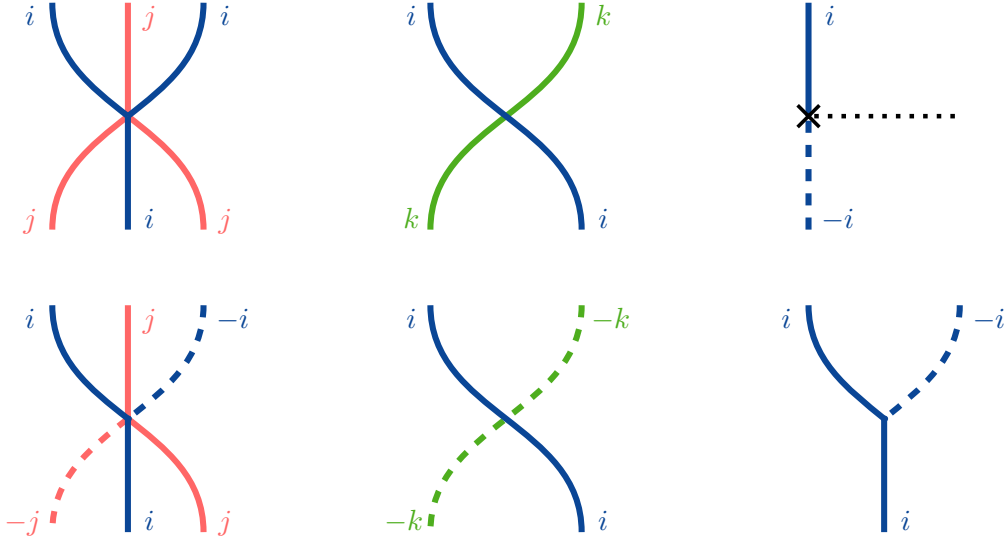
		\caption{An expanded set of permitted vertices involving frayed strands.}
		\label{fig:frayedweavevertices}
	\end{figure}
	\begin{figure}[htbp]
		\centering
		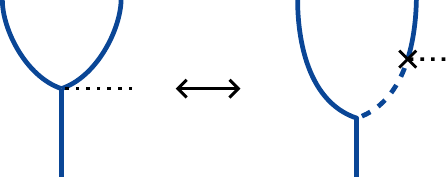
		\caption{Factoring a classical trivalent vertex as a fraying followed by a twining.}
		\label{fig:factoringwithfrayings}
	\end{figure}
	In Figure \ref{fig:frayedweavevertices}, we will call the top-right vertex a \emph{fraying}. The bottom-right vertex will be called a \emph{twining}. Similar to the trivalent vertices in Figure \ref{fig:classicalweavevertices}, fraying vertices have a rightward raking ray which modifies the variables of strands to the right. Trivalent vertices of the previous section will be replaced with a fraying followed by a twining (see Figure \ref{fig:factoringwithfrayings}), but we can twine frayed strands which did not originate from a fraying. The associated maps between charts for the new vertices correspond to the following matrix identities:
	\begin{align}
		B_i(z_1) &= u_{-i}(z_1^{-1})\alpha^\vee_i(z_1)u_{i}(-z_1^{-1}),\label{eqn:frayingvertex}\\
		B_i(z_1)B_j(z_2)u_{-i}(z_3) &= u_-j(z_3)B_i(z_1-z_2z_3)B_j(z_2), \label{eqn:leftward3}\\ B_i(z_1)u_{-k}(z_2) &=  u_{-k}(z_2)B_i(z_1),\,\; \text{ and} \label{eqn:leftward2}\\ B_i(z_1)u_{-i}(z_2) &=  B_i(z_1-z_2).\label{eqn:twiningvertex}
	\end{align}
	Observe that we can recover (\ref{eqn:trivalentidentity}) from (\ref{eqn:frayingvertex}) and (\ref{eqn:twiningvertex}).
	\begin{rem}
		The difference between the fraying vertex corresponding to (\ref{eqn:frayingvertex}) and the opening crossing process from \citep[\S2.3]{CGGS24}
		is that we do not (yet) propagate the lower triangular matrix leftwards. This gives the transition map between adjacent charts $\cO^{(\dots,i,\dots)}\dasharrow \cO^{(\dots,-i,\dots)}$.
	\end{rem}
	
	\begin{rem}
		The recipe to construct the quiver for any Demazure weave uses the \emph{Lusztig cycles} on a weave, and their intersection multiplicities. In \citep{CGGLSS25}, these described by lifting the matrix identities to \emph{framed flags} in $G/U$. The Lusztig cycles were categorified in \citep[Definition~4.1]{CC26} as a simple-minded collection of a particular category.
		The matrix identities associated to our frayed weave vertices can be lifted to framed flags in a similar fashion. It would be interesting to find the analogue of Lusztig cycles and their categorification.
	\end{rem}
	\begin{prop}\label{prop:fraysleftwards}
		Suppose $i$ and $j$ are adjacent and $i$ and $k$ are not. We have Poisson maps between $2$- and $3$-dimensional Bott--Samelson charts given by
		\begin{align*}
			\cO^{(i,-k)} &\to \cO^{(-k,i)},\quad \Phi_{(i,-k)}(z_1,z_2)\mapsto \Phi_{(-k,i)}(z_2,z_1) \qquad \text{ and}\\
			\cO^{(i,j,-i)} &\to \cO^{(-j,i,j)}, \quad \Phi_{(i,j,-i)}(z_1,z_2,z_3) \mapsto \Phi_{(-j,i,j)}(z_3,z_1-z_2z_3, z_2).
		\end{align*}
	\end{prop}
	\begin{proof}
		The proof proceeds identically to that of Proposition \ref{prop:Poissonmapsbetweenshortcells}. As before,  the $2$-dimensional case holds as the Poisson bracket is trivial on both the domain and codomain.
		For $\gamma=(i,j,-i)$ and $\widetilde{\gamma}=(-j,i,j)$, and let $f:\cO^{\gamma} \to \cO^{\widetilde{\gamma}}$ denote the map defined above.
		We can use the Elek--Lu formula to calculate the brackets on $\cO^{\gamma}$ and $\cO^{\widetilde{\gamma}}$ as
		\begin{equation*}
			\begin{array}{l l l}
				\{z_1,z_2\}_{\gamma} = -z_1z_2,& \{z_1,z_3\}_{\gamma} = -z_1z_3 +2z_2z_3^2,& \{z_2,z_3\}_{\gamma}=z_2z_3, \quad \text{and}\\[4pt]
				\{\widetilde{z}_1,\widetilde{z}_2\}_{\widetilde{\gamma}} = \widetilde{z}_1\widetilde{z}_2,& \{\widetilde{z}_1,\widetilde{z}_3\}_{\widetilde{\gamma}} =-\widetilde{z}_1\widetilde{z}_3,& \{\widetilde{z}_2,\widetilde{z}_3\}_{\widetilde{\gamma}} =-\widetilde{z}_2\widetilde{z}_3.
			\end{array}
		\end{equation*}
		That $f:\cO^{\gamma}\to \cO^{\widetilde{\gamma}}$ is Poisson can now be explicitly checked:
		\begin{align*}
			f^*\left(\{\widetilde{z}_1,\widetilde{z}_2\}_{\widetilde{\gamma}}\right) &= f^*(\widetilde{z}_1\widetilde{z}_2) = z_3(z_1-z_2z_3)= z_1z_3 - 2z_2z_3^2 + z_2z_3^3\\&= \{z_3,z_1-z_2z_3\}_{\gamma} = \{f^*(\widetilde{z}_1),f^*(\widetilde{z}_2)\}_\gamma \, ,\\
			f^*\left(\{\widetilde{z}_1,\widetilde{z}_3\}_{\widetilde{\gamma}}\right) &= f^*(-\widetilde{z}_1\widetilde{z}_3)= - z_3z_2\\
			&= \{z_3,z_2\}_{\gamma} = \{f^*(\widetilde{z}_1),f^*(\widetilde{z}_3)\}_\gamma\,, \; \text{ and}\\
			f^*\left(\{\widetilde{z}_2,\widetilde{z}_3\}_{\widetilde{\gamma}}\right) &= f^*(-\widetilde{z}_2\widetilde{z}_3) = -(z_1-z_2z_3)z_2= -z_1z_2 + z_2^2z_3\\&= \{z_1-z_2z_3,z_2\}_{\gamma} = \{f^*(\widetilde{z}_2),f^*(\widetilde{z}_3)\}_\gamma \,. \qedhere
		\end{align*}
	\end{proof}
	
	\begin{rem}
		The map $\cO^{(i,j,-i)}\to \cO^{(-j,i,j)}$ is an isomorphism, and one can define a frayed weave vertex for its inverse. There are several more matrix identities giving vertices with frayed strands. The $6$ vertices in Figure \ref{fig:frayedweavevertices} are sufficient for the arguments in this paper. Studying the combinatorics of frayed weaves and their equivalences is left to future work.
	\end{rem}
	
	Both these above maps respect the multiplication maps on these charts. Therefore the argument in Proposition \ref{prop:Poissonmapsbetweencells} can be repeated, and we can apply these moves within a signed expression to ``bring the negative letter leftward''.
	The following lemma gives some equivalences for these moves. They are diagrammatically shown in Figures \ref{fig:commutingfrayedweave} and \ref{fig:frayedweaveequivlabelled}.
	\begin{lem}\label{lem:leftwardfrayedequivalences}
		When $i,j$ are adjacent roots, the compositions $$\cO^{(i,j,i,-j)}\to \cO^{(i,-i,j,i)} \to \cO^{(i,j,i)} \quad \text{and} \quad \cO^{(i,j,i,-j)} \to \cO^{(j,i,j,-j)} \to \cO^{j,i,j} \to \cO^{(i,j,i)}$$ give the same Poisson map between the Bott--Samelson charts $\cO^{(i,j,i,-j)}\to \cO^{(i,j,i)}$.
		
		Similarly, for non-adjacent roots $i$ and $k$, the compositions $$\cO^{(i,k,-i)} \to \cO^{(k,i,-i)}\to \cO^{(k,i)} \to \cO^{(i,k)}\quad \text{and} \quad\cO^{(i,k,-i)} \to \cO^{(i,-i,k)} \to \cO^{(i,k)}$$ give the same Poisson map $\cO^{(i,k,-i)} \to \cO^{(i,k)}$.
	\end{lem}
	\begin{proof}
		This is checked by composing the relevant maps, see also Figures \ref{fig:commutingfrayedweave} and \ref{fig:frayedweaveequivlabelled}.
	\end{proof}

	\begin{figure}[htbp]
		\centering
		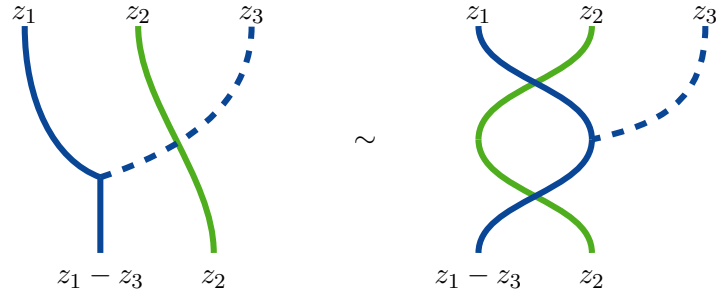
		\caption{Equivalent frayed weaves from $(i,k,-i)$ to $(i,k)$.}
		\label{fig:commutingfrayedweave}
	\end{figure}

	\begin{figure}[htbp]
		\centering
		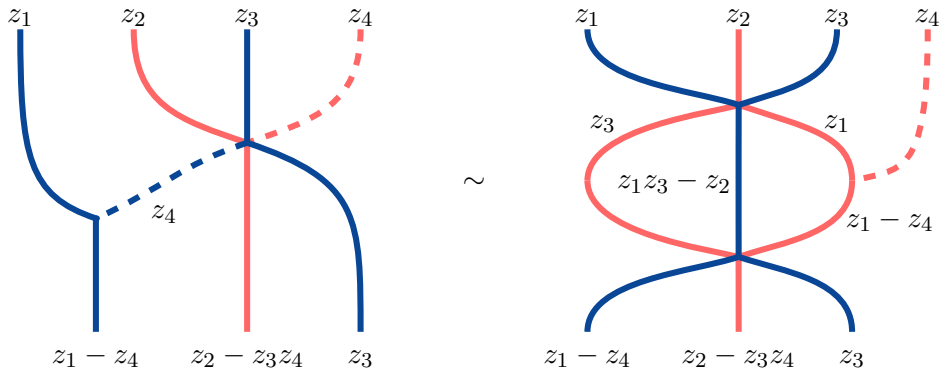
		\caption{Equivalent frayed weaves from $(i,j,i,-j)$ to $(i,j,i)$.}
		\label{fig:frayedweaveequivlabelled}
	\end{figure}

	\begin{cor}\label{cor:twinepastlongestelt}
		In simply-laced type, we can apply braid moves and the vertices from Proposition \ref{prop:fraysleftwards} to get a weave $(\underline{w_0}, -i) \to (i^*,-i^*,\underline{v})$, where $v=w_0s_i$ and $i^*$ is such that $s_{i^*}=w_0 s_i w_0$.
	\end{cor}
	\begin{proof}
		From any reduced word for $w_0$, we use a sequence of braid moves and commutations to get to a word $\underline{w_0}$ which ends with $i$. Apply the trivalent twining vertex, and then apply the reversed sequence of braid moves and commutators to ``undo'' the coordinate change. Applying Lemma \ref{lem:leftwardfrayedequivalences} allow us to rewrite our weave as a sequence of leftward moves.
	\end{proof}
	
	\begin{rem}
		In type $A$, a direct proof is possible by using a good choice of reduced word $\underline{w_0}$ and appealing to a braid group argument. See for example \citep[Lemma~2]{Gar69}.
	\end{rem}
	
	\begin{prop}\label{prop:twinablestrands}
		For a signed expression of the form $(w_0,\gamma)$, the leftmost frayed strand corresponding to the leftmost negative letter in $\gamma$ can be twined by bringing it leftwards using an appropriate sequence of frayed weave moves.
	\end{prop}
	\begin{proof}
		We use  the procedure in \citep[Definition 5.15]{CGGS24}, where one chooses a word for $w_0$ ending with $i$ and then twining them.
		Suppose $\gamma = (\underline{u},-i,\gamma'')$, where $\underline{u}$ is a sequence of positive letters. Pick a reduced word for $\underline{w_0}$ for $w_0$ which has $(\underline{u})^*$ as a prefix. Then there is a sequence of braid moves between positive braids $(\underline{u},\underline{w_0})$ and $((\underline{u})^*,\underline{w_0})$. By Corollary \ref{cor:twinepastlongestelt}, we can bring the negative letter leftwards and twine it with the first letter in $\underline{w_0}$. Undoing the previous sequence of braid moves brings us back to $(\underline{w_0},\underline{u},\gamma'')$.
	\end{proof}
	
	\begin{rem}
		This procedure is similar to the opening crossing procedure which appears in \citep[\S2.3, see also Theorem 5.18]{CGGS24}. The difference here is that we might already begin with a lower triangular factor to bring leftward which did not originate from fraying a strand.
	\end{rem}
	\begin{prop}\label{prop:mapstosupportedsubword}
		Let $\gamma$ be a subexpression of $\underline{w}$ and let $\gamma^+ = (\gamma_i)_{\gamma_i >0}$. Then the Bott--Samelson chart $\cO^\gamma$ is isomorphic to $\cO^{\gamma^{+}} \times \mathbb{A}^{m}$, where $m$ is the number of unsupported letters in $\gamma$.
		Endowing the product with a Poisson bracket which is log-canonical with respect to the Bott--Samelson coordinates, this becomes a Poisson isomorphism.
	\end{prop}
	
	\begin{proof}
		From the isomorphism in Lemma \ref{lem:braidopeninchart}, we can replace the open part of the Bott--Samelson cell with the braid variety by left concatenation by the longest element $w_0$. Using our braid moves, we can iteratively bring the left-most frayed strand as far left as possible, before twining them with a strand of the same colour as in Lemma \ref{lem:frayedtwiningisPoisson} and Proposition \ref{prop:twinablestrands} to obtain a log-canonical product.
		This gives us an isomorphism on the open subsets $$\cO^{\gamma}\cap m^{-1}(B^-B/B) \cong (\cO^{\gamma^+}\cap m^{-1}(B^-B/B))\times \mathbb{A}^m,$$ and the isomorphism can be extended to the whole spaces $\cO^{\gamma} \cong \cO^{\gamma^+}\times \mathbb{A}^m$ since each of the component maps are regular in the Bott--Samelson coordinates.
	\end{proof}
	
	Let $\rho:\cO^{\gamma}\to \cO^{\gamma^+}$ be the composition of the above isomorphism with the projection to the first factor.
	We define the seed $\seed(\gamma)$ to be the pullback of $\seed(\gamma^+)$ by $\rho$ together with isolated frozens corresponding to the unsupported letters. Explicitly, if $j_1<\dots<j_m$ are the unsupported positions in $\gamma$, the seed $\seed(\gamma) := \rho^*(\seed(\gamma^+)) \sqcup \left\{\boxed{z_{j_l}}\right\}_{l=1,\dots,m}$ has:
	\begin{itemize}
		\item Cluster variables $(\rho^*(A_j))_{A_j\in \seed(\gamma^+)}$ together with the coordinates $(z_{j_k})_{l=1}^m$.
		\item Frozen variables which are pullbacks of frozens from $\seed(\gamma^+)$ together with the coordinates $(z_{j_k})_{l=1}^m$.
		\item Exchange matrix given by that from $\seed(\gamma^+)$ extended by zeroes.
	\end{itemize}
	\begin{thm}\label{thm:BSchartshaveclusterstructure}
		For $G$ is simply-laced, the seed $\seed(\gamma)$ endows the Bott--Samelson chart $\cO^\gamma$ with a cluster algebra structure
		$\cA(\seed(\gamma),\emptyset)=\C[\cO^\gamma]$,
		which is compatible with the standard Poisson structure.
	\end{thm}
	\begin{proof}
		Proposition \ref{prop:mapstosupportedsubword} gives the isomorphism
		$\C[\cO^\gamma]\cong \C[\cO^{\gamma^+}][z_{i_k}]_{k=1}^m$.
		By \citep{SW21} or \citep{CGGLSS25}, we have that $\cA(\seed(\gamma^+),\emptyset) = \C[\cO^{\gamma^+}]$. Hence as subalgebras of the field $\C(\cO^{\gamma})$, we have
		$$\cA(\seed(\gamma),\emptyset) = \cA(\rho^*(\seed(\gamma^+)),\emptyset)[z_{i_k}]_{k=1}^m = \rho^*(\C[\cO^{\gamma^+}])[z_{i_k}]_{k=1}^m = \C[\cO^{\gamma}].$$
		Compatibility with the standard Poisson structure follows because	the cluster structure on $\cO^{\gamma^+}$ is compatible with the standard Poisson structure by Corollary \ref{cor:ClusterPoissonCompatible} and the map in Proposition  \ref{prop:mapstosupportedsubword} is a Poisson isomorphism with a log-canonical product.
	\end{proof}
	
	\begin{cor}
		The projection map $\rho:\cO^\gamma \to \cO^{\gamma^+}$ is a quasi-cluster morphism.
	\end{cor}
	\begin{proof}
		The algebra map is given by embedding the seed $\seed(\gamma^+)$ into $\rho^*(\seed(\gamma^+))\subseteq\seed(\gamma)$. The statement follows as the extra isolated frozens do not affect the exchange ratios.
	\end{proof}

	\begin{lem}\label{lem:sectionissection}
		The embedding $\sigma: \cO^{\underline{w}[\ell-1]} \hookrightarrow \cO^{\underline{w}}$ is a section of the projection map $\rho$.
	\end{lem}
	\begin{proof}
		Denote the Bott--Samelson coordinates on $\cO^{\underline{w}}$ and $\cO^{\underline{w}[\ell-1]}$ by $z_i$ and $\widetilde{z_i}$ respectively. By construction, $\rho^*(\widetilde{z_j}) = z_j - p_j(z_{j+1},\dots, z_{\ell-1})\cdot z_\ell$, for polynomials $p_j\in \C[z_{j+1},\dots,z_{\ell-1}]$.
		We have
		$\sigma^*(z_i)=\widetilde{z}_i$ if $i<\ell$ and $\sigma^*(z_\ell)=0$, and therefore $\sigma^*\circ\rho^* =\mathrm{id}_{\C[\cO^{\underline{w}[\ell-1]}]}$.
	\end{proof}
	
	\begin{rem}
		We emphasise that this morphism $\rho: \cO^{\underline{w}}\to \cO^{\underline{w}[\ell-1]}$ is \emph{not} the truncation morphism $\psi_{\underline{w}}$ despite $\sigma$ being a section of both $\rho$ and $\psi_{\underline{w}}$.
		See also Example \ref{eg:mutationsequences(1,2,1,-2)}.
	\end{rem}

	\begin{eg}\label{eg:(2,3,1,-2,2,-1)}
		In type $A$, let $\gamma = (2,3,1,-2,2,-1)$ so that $\gamma^+= (2,3,1,2)$. For Bott--Samelson coordinates $z_1,\dots,z_{6}$ on $\cO^\gamma$, applying the above procedure to the first negative letter, $-2$ in the fourth position, yields the map
		$$\cO^{(2,3,1,-2,2,-1)}\to \cO^{(2,3,1,2,-1)},\quad (z_1,z_2,z_3,z_4,z_5,z_6) \mapsto (z_1-z_2z_3z_4, z_2,z_3,z_5,z_6).$$
		Repeating for the next negative letter, $-1$ in the fifth (formerly sixth) position, yields
		$$\cO^{(2,3,1,2,-1)}\to \cO^{(2,3,1,2)},\quad(z_1',z_2',z_3',z_4',z_5') \mapsto (z_1'-z_2'z_5', z_2',z_3'-z_4'z_5',z_5').$$
		These compose to get the map
		$$\rho:\cO^{\gamma}\to \cO^{\gamma^+}, \quad (z_1,z_2,z_3,z_4,z_5,z_6) \mapsto (z_1-z_2z_3z_4-z_2z_6,z_2,z_3-z_5z_6,z_5).$$
		
		A weave representing this is shown in Figure \ref{fig:reductiontosupport}. Continuing this weave down to the Demazure product of $w_0$ would give us a log-canonical toric chart in $\cO^{\gamma}\cap m^{-1}(B^-B/B)$. The Lusztig cycle procedure from \citep{CGGLSS25} would give us a quiver for the corresponding toric chart in $\cO^{\gamma}\cap m^{-1}(B^-B/B) \subset \cO^{\gamma}$. In this way, pulling back the cluster structure by $\rho$ is the precomposition of a weave $(\underline{w_0},\gamma^+)\to \underline{w_0}$ by our ``frayed weave'' $(\underline{w}_0,\gamma)\to(\underline{w_0},\gamma^+)$.
		The seed $\seed(\gamma^+)$ on $\cO^{\gamma^+} = \cO^{(2,3,1,2)}$  corresponding 
		to the right inductive weave
		is:
		\[\begin{tikzcd}[sep=tiny]
			&& {\boxed{z_2}} & \\
			{z_1} &&& {\boxed{z_1z_4-z_2z_3}} \\
			&& {\boxed{z_3}}
			\arrow[dashed, from=1-3, to=2-4]
			\arrow[from=2-1, to=1-3]
			\arrow[from=2-4, to=2-1]
			\arrow[dashed, from=3-3, to=2-4]
			\arrow[from=2-1, to=3-3]
		\end{tikzcd}\]
		where the boxed functions denote frozen variables, and the dashed arrows are half-weighted.
		For comparison, the corresponding seed $\seed(\gamma) = \rho^*(\seed(\gamma^+)) \sqcup \left\{\boxed{z_4},\boxed{z_6}\right\}$ on $\cO^\gamma$ is:
		\[\begin{tikzcd}[sep=tiny]
			&&& {\boxed{z_2}} & \\
			{\boxed{z_4}} && {z_1-z_2z_3z_4-z_5z_6} && {\boxed{(z_1-z_2z_3z_4-z_2z_6)z_5-z_2(z_3-z_5z_6)}} \\
			{\boxed{z_6}} &&& {\boxed{z_3-z_5z_6}}
			\arrow[dashed, from=1-4, to=2-5]
			\arrow[from=2-3, to=1-4]
			\arrow[from=2-5, to=2-3]
			\arrow[from=2-3, to=3-4]
			\arrow[dashed, from=3-4, to=2-5]
		\end{tikzcd}\]
	\end{eg}
	
	\begin{figure}[htbp]
		\centering
		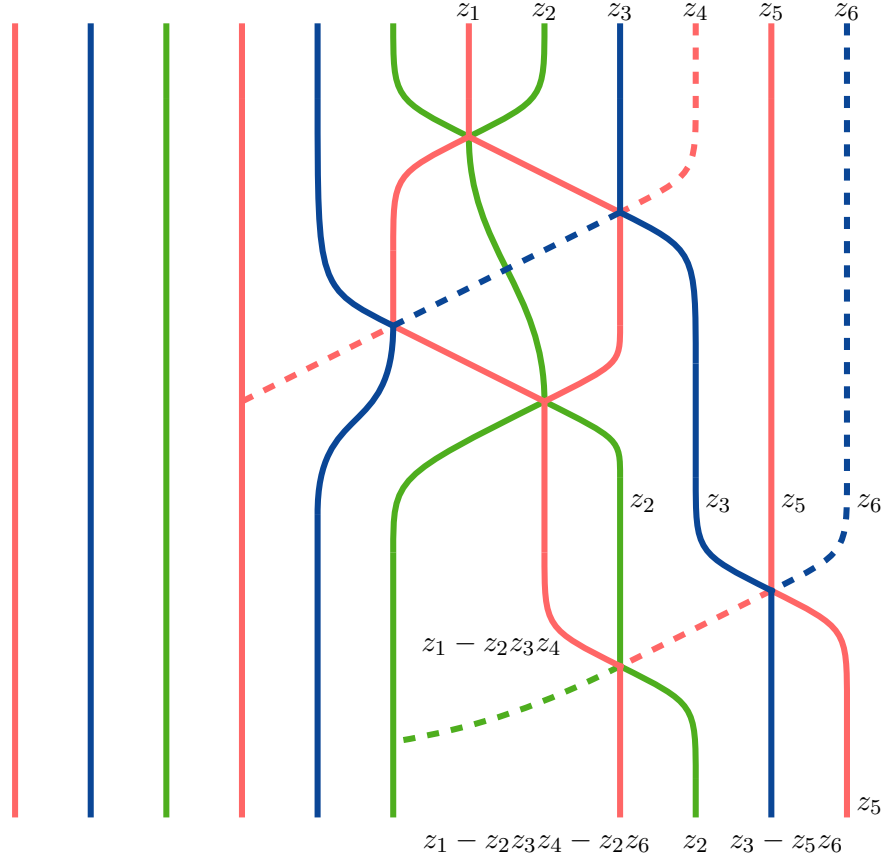
		\caption{A frayed weave $(\underline{w_0},2,3,1,-2,2,-1) \to (\underline{w_0},2,3,1,2)$ representing the map $\rho:\cO^{(2,3,1,-2,2,-1)} \to \cO^{(2,3,1,2)}$. For convenience, we've chosen the reduced word $\underline{w_0}=(2,1,3,2,1,3)$ for the longest element. The map does not depend on the coordinates for those positions, nor on $\underline{w_0}$.}
		\label{fig:reductiontosupport}
	\end{figure}

	\begin{rem}
		There has been some study of Legendrian links and braid varieties with negative crossings (see \citep[\S3.4]{CGGS26}). It is unclear to the author if there is any relation to our frayed strands, or whether there is another topological interpretation of frayed strands.
	\end{rem}

	\subsection{Non-simply-laced types}\label{sec:FoldedTypes}
	Following \citep{CGGLSS25}, we extend the notion of frayed weaves to non-simply-laced types by including new vertices of higher valency.
	These correspond to maps between Bott--Samelson charts supported on a root subsystem of types $B_2$ or $G_2$.
	As in \citep[\S6]{CGGLSS25}, higher valency vertices are foldings of simply-laced weaves, with type $B_2$ folded from $A_3$ and type $G_2$ folded from $D_4$ to obtain $8$-valent and $12$-valent vertices respectively. We state the relations for $B_2$ and draw the corresponding weaves. Those for $12$-valent vertices are analogous.	Note the lengths of the roots affects the strand label propagation.
	
	\begin{lem}\label{lem:B2frayedvertices0}
		On a type $B_2$ root subsystem with Cartan submatrix $\begin{pmatrix} 2& -1 \\ -2 &2 \end{pmatrix}$, we have:
		\begin{align}
			B_{(1,2,1,2)}(\underline{z}) &= B_2(z_4)B_1(z_1z_4-z_3)B_2(z_1^2z_4-2z_1z_3+z_2)B_1(z_1),\\
			B_{(2,1,2,1)}(\underline{z}) &= B_1(z_4)B_2(z_1z_4^2-2z_2+z_3))B_1(z_1z_4-z_2)B_2(z_1),\\
			B_{(1,2,1,-2)}(\underline{z}) &= u_{-2}(z_4)B_1(z_1-z_3z_4)B_2(z_2-z_3^2z_4)B_1(z_3),\\
			B_{(2,1,2,-1)}(\underline{z}) &= u_{-1}(z_4)B_2(z_1-2z_2z_4+z_3z_4^2)B_1(z_2-z_3z_4)B_2(z_3).
		\end{align}
		The corresponding $8$-valent vertices and their unfolded weaves are depicted in Figure \ref{fig:B2folding2}.
	\end{lem}
	\begin{proof}
		These relations can seen by applying the previous simply-laced identities to an unfolded weave for type $B_2$, see Figure \ref{fig:B2folding2}.
		The formulas have been checked by an $SL_4$ computation in SAGE.
	\end{proof}
	
	\begin{figure}[htbp]
		\centering
		\def\svgwidth{\linewidth}
		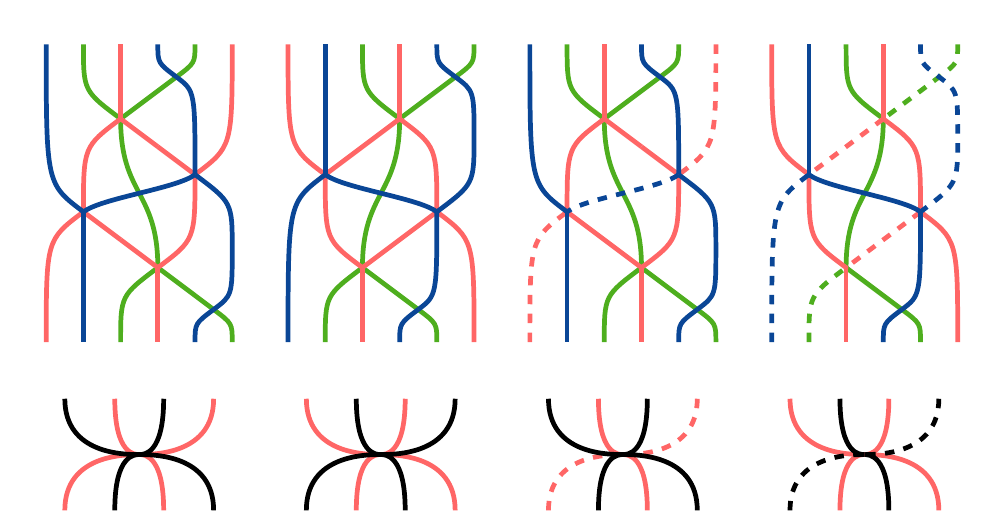
		\caption{
			Four frayed weaves in type $A_3$ and the $8$-valent vertices they fold to in type $B_2$.
		}
		\label{fig:B2folding2}
	\end{figure}
	
	\begin{lem}
		When $i,j$ form a $B_2$ root subsystem, the compositions
		\begin{align*}
			\cO^{(i,j,i,j,-i)} \to \cO^{(i,-i,j,i,j)} \to \cO^{(i,j,i,j)} \;\text{ and }\; \cO^{(i,j,i,j,-i)} \to \cO^{(j,i,j,i,-i)}\to \cO^{(j,i,j,i)}\to \cO^{(i,j,i,j)}
		\end{align*}
		give the same Poisson map between the Bott--Samelson charts $\cO^{(i,j,i,j,-i)}\to \cO^{(i,j,i,j)}$.
	\end{lem}
	\begin{proof}
		This is analogous the first statement in Lemma \ref{lem:leftwardfrayedequivalences}, and can be proved by lifting the relations to an  appropriate unfolding weave and iteratively applying Lemma \ref{lem:leftwardfrayedequivalences}.
	\end{proof}
	
	\begin{cor}
		Theorem \ref{thm:BSchartshaveclusterstructure} holds when $G$ is any semisimple algebraic group.
	\end{cor}
	
	\begin{eg}\label{eg:B2example}
		In type $B_2$,
		consider the signed expression $\gamma = (1,2,1,-2,2)$, so that $\gamma^+=(1,2,1,2)$. In the Bott--Samelson coordinates $z_i$ for $\cO^{\gamma^+}$, $E(\gamma^+)$ and $\seed(\gamma^+)$ are:
		$$\begin{pmatrix}
			0 & 2 & -1 & 0\\
			-1 & 0 & 1 &-1\\
			1 & -2 & 0 & 1\\
			0 & 1 & -\tfrac{1}{2}& 0
		\end{pmatrix} \quad \text{and} \quad \begin{tikzcd}[sep=small]
			& {z_2} && {\boxed{z_2z_4-z_3^2}} \\
			{z_1} && {\boxed{z_1z_3-z_2}}
			\arrow["{2:1}", from=1-2, to=2-3]
			\arrow[from=1-4, to=1-2]
			\arrow["{1:2}", from=2-1, to=1-2]
			\arrow["{1:2}", dashed, from=2-3, to=1-4]
			\arrow[from=2-3, to=2-1]
		\end{tikzcd}$$
		Denoting the coordinates for $\cO^\gamma$ by $z_i'$, the identities in Lemma \ref{lem:B2frayedvertices0} give the map $$\rho^*:\C[\cO^{\gamma^+}] \to \C[\cO^\gamma], \quad (z_1,z_2,z_3,z_4) \mapsto (z_1'-z_3'z_4',z_2'-(z_3')^2z_4',z_3',z_5').$$
		Then the seed $\seed(\gamma)$ is then:
		\[\begin{tikzcd}[sep=small]
			{\boxed{z_4'}} && {z_2'-(z_3')^2z_4'} && {\boxed{(z_2'-(z_3')^2z_4')z_5'-(z'_3)^2}} \\
			& {z'_1-z'_3z'_4} && {\boxed{z'_1z'_3-z'_2}}
			\arrow["{{2:1}}", from=1-3, to=2-4]
			\arrow[from=1-5, to=1-3]
			\arrow["{{1:2}}", from=2-2, to=1-3]
			\arrow["{{1:2}}", dashed, from=2-4, to=1-5]
			\arrow[from=2-4, to=2-2]
		\end{tikzcd}\]
		Mutating at the first position yields the cluster variable $z_3' = \rho^*(z_3)$, and further mutating at the second position yields the cluster variable $z_5'=\rho^*(z_4)$.
		Together with the initial variables $z_4'$, $z_1'-z_3'z_4'$, and $z_2'-(z_3')^2z_4'$ these generate the polynomial ring $\C[\cO^\gamma]$.
	\end{eg}
	
	\begin{rem}\label{rem:LuYakimovCGLknown}
		In unpublished work, it was observed by Jiang-Hua Lu and Milen Yakimov that the Bott--Samelson coordinates for any $\cO^{\gamma}$ are, up to a polynomial transformation, isomorphic to a symmetric Poisson CGL extension. This geometric interpretation using weaves to construct the isomorphism $\cO^\gamma\cong\cO^{\gamma^+}\times \mathbb{A}^m$ by factoring the maps through several other Bott--Samelson charts is new, and recovers their change of coordinates.
	\end{rem}

	\section{Compatibility with transition functions}\label{sec:clustertransitions}
	\subsection{Mutation sequences and pulling back seeds}
	\begin{defn}\label{defn:levelmutation}
		For each position $1\leq j\leq \ell$, we define a sequence of mutations. Let $i=i_j$ be the $a_j$-th appearance of $i$ in $\underline{w}$ and let $n_i$ be the total number of appearances of $i$ in $\underline{w}$. Then define the composition $			\overrightarrow{\mu_{\underline{w}}} := \widetilde{\mu}_\ell \circ \cdots \circ \widetilde{\mu}_2 \circ \widetilde{\mu}_1$ with the factors
		$\widetilde{\mu}_{j}$ given by
		\begin{equation}
			\widetilde{\mu}_{j} = \begin{cases}
				\mu_{(i,n_i-a_j)}\circ\cdots \circ \mu_{(i,1)} & \text{ if } a_j<n_i,\\
				\mathrm{id} & \text{ if } a_j=n_i.
			\end{cases}
		\end{equation}
		This sequence mutates in ascending order along the level of each letter, and skips some final mutations dependent on the number of times we have already mutated at this level.
	\end{defn}
	By \citep[\S13.1]{GLS11} or repeated applications of \citep[Proposition 3.25(1)]{SW21} (see also \citep[Proposition 4.15]{BY25}), the sequence $\overrightarrow{\mu_{\underline{w}}}$
	takes the seed $\seed(\underline{w})$ to the reversed seed $\seed(\overline{i_\ell}, \overline{i_{\ell-1}},\dots, \overline{i_2},\overline{i_1})$.
	Its cluster variables are the generalised minors of the reversed partial products:
	\begin{equation}\label{eqn:reversedseedvariables}
		\widetilde{A_j} = \Delta_{\omega_{i_{\ell-j+1}},\omega_{i_{\ell-j+1}}} (B_{\underline{w}[\ell-j+1,\ell]}(z_{\ell-j+1},\dots,z_\ell)).
	\end{equation}
	In particular, the first cluster variable $\widetilde{A_1} = z_\ell$ is the last Bott--Samelson coordinate. We refer to \citep[\S2.7]{BY25} for the construction of the exchange matrix $E(\overline{i_\ell}, \overline{i_{\ell-1}},\dots, \overline{i_2},\overline{i_1}) = -E(i_\ell,\dots,i_1)$.
	
	\begin{defn}\label{defn:reversedlevelmutation}
		Define the composition $\overrightarrow{\mu'_{\underline{w}}}= \widehat{\mu}_{\ell} \circ \cdots\circ \widehat{\mu}_{2} \circ \widehat{\mu}_{1}$, where
		\begin{equation}
			\widehat{\mu}_{j} = \begin{cases}
				\mu_{(i,1)} \circ\cdots \circ \mu_{(i,a_j-1)}  & \text{ if } 1<a_j,\\
				\mathrm{id} & \text{ if } a_j=1.
			\end{cases}
		\end{equation}
	\end{defn}
	
	\begin{lem}\label{lem:reversingmutationsequences}
		The two seeds $\overrightarrow{\mu_{\underline{w}}}(\seed(\underline{w}))$ and $\overrightarrow{\mu'_{\underline{w}}}(\seed(\underline{w}))$
		are both equal to $\seed(\overline{i_\ell}, \overline{i_{\ell-1}},\dots, \overline{i_2},\overline{i_1})$.
	\end{lem}
	\begin{proof}
		Inductively assuming the statement for the truncated word $\underline{w}[\ell-1]$, we have
		\begin{align*}
			\overrightarrow{\mu'_{\underline{w}}}(\seed(\underline{w})) &= \widehat{\mu}_{\ell} \circ \overrightarrow{\mu_{\underline{w}[\ell-1]}}(\seed(\underline{w}))\\
			&= \widehat{\mu}_{\ell}  \left(\seed(\overline{i_{\ell-1}},\dots,\overline{i_2},\overline{i_1},i_\ell)\right)\\
			&= \widehat{\mu}_{\ell} \circ (\mu_{(i_{\ell},n_{i_l}-1)}\circ \dots \circ \mu_{(i_{\ell},1)}) \left(\seed(\overline{i_\ell},\overline{i_{\ell-1}},\dots,\overline{i_2},\overline{i_1})\right)\\
			&= \seed(\overline{i_\ell},\overline{i_{\ell-1}},\dots,\overline{i_2},\overline{i_1}).\qedhere
		\end{align*}
	\end{proof}
	\begin{rem}
		It is also possible to prove the coincidence of the two seeds by commuting mutations of non-adjacent vertices. Alternatively, one checks that both sequences are layered $T$-systems in the sense of Yakimov \citep{Yak26} and hence both are maximal green sequences, see also \citep[Theorem 4.1]{SW21}.
	\end{rem}
	\begin{rem}
		The sequence $\overrightarrow{\mu_{\underline{w}}}$ has a nice interpretation in terms of triangulations in configurations of flags, see \citep[Proposition 4.2]{SW21}.
		The sequence $\overrightarrow{\mu'_{\underline{w}}}$ can be interpreted as doing the same but in the opposite direction, iteratively swapping letters $i$ to their opposites $\overline{i}$ and bringing them leftward instead of rightward.
	\end{rem}
	
	\begin{eg}\label{eg:(1,2,1,2,1,2)}
		We illustrate this sequence with an example. In type $A_2$, let $\underline{w} = (1,2,1,2,1,2)$ and $\gamma= (1,2,1,2,1,-2)$. The quiver for the seed $\seed(\underline{w})$ looks like
		\[\begin{tikzcd}[sep=small]
			& 2 && 4 && {\boxed{6}} \\
			1 && 3 && {\boxed{5}}
			\arrow[from=1-2, to=2-3]
			\arrow[from=1-4, to=1-2]
			\arrow[from=1-4, to=2-5]
			\arrow[from=1-6, to=1-4]
			\arrow[from=2-1, to=1-2]
			\arrow[from=2-3, to=1-4]
			\arrow[from=2-3, to=2-1]
			\arrow[dashed, from=2-5, to=1-6]
			\arrow[from=2-5, to=2-3]
		\end{tikzcd}\]
		Calculating generalised minors for the partial products, the cluster is given by
		\begin{align*}
			A_1&= z_1, \quad A_2 = z_2, \quad A_3 = z_1z_3-z_2, \quad A_4 = z_2z_4-z_3,\\
			A_5&=z_1z_3z_5-z_2z_5-z_1z_4+1,\quad A_6 = z_2z_4z_6-z_3z_6-z_2z_5+1.
		\end{align*}
		Our two forms of the reversing sequence are $\overrightarrow{\mu_{\underline{w}}}=\mathrm{id}\circ\mathrm{id}\circ (\mu_2)\circ (\mu_1) \circ (\mu_4\circ\mu_2) \circ (\mu_3\circ\mu_1)$ and $\overrightarrow{\mu'_{\underline{w}}}=\mathrm{id}\circ\mathrm{id}\circ (\mu_2\circ \mu_4)\circ (\mu_1\circ \mu_3) \circ (\mu_2) \circ (\mu_1)$  Applying  either  to $\seed(\underline{w})$ yields:
		\[\begin{tikzcd}[sep=small]
			& 2 && 4 && {\boxed{6}} \\
			1 && 3 && {\boxed{5}}
			\arrow[from=1-2, to=1-4]
			\arrow[from=1-4, to=1-6]
			\arrow[from=1-4, to=2-1]
			\arrow[from=1-6, to=2-3]
			\arrow[from=2-1, to=1-2]
			\arrow[from=2-1, to=2-3]
			\arrow[from=2-3, to=1-4]
			\arrow[from=2-3, to=2-5]
			\arrow[dashed, from=2-5, to=1-6]
		\end{tikzcd} \quad \sim \quad
		\begin{tikzcd}[sep=small]
			2 && 4 && {\boxed{6}} & \\
			& 1 && 3 && {\boxed{5}}
			\arrow[from=1-1, to=1-3]
			\arrow[from=1-3, to=1-5]
			\arrow[from=1-3, to=2-2]
			\arrow[from=1-5, to=2-4]
			\arrow[from=2-2, to=1-1]
			\arrow[from=2-2, to=2-4]
			\arrow[from=2-4, to=1-3]
			\arrow[from=2-4, to=2-6]
			\arrow[dashed, from=2-6, to=1-5]
		\end{tikzcd}\]
		with mutable variables $A_1= z_5$, $A_2 = z_6$, $A_3 = z_3z_5-z_4$, and $A_4 = z_4z_6-z_5$.
	\end{eg}
	
	\begin{lem}\label{lem:pullbackvars}
		Let $\gamma = \underline{w}(\ell)$ be the subexpression obtained by negating the last letter.
		Let $(A_j^+)_{j=1}^{\ell-1}$ be the cluster variables of the  seed $\overrightarrow{\mu_{\underline{w}[\ell-1]}}(\seed(\underline{w}[\ell-1]))$ on $\cO^{\gamma^+}$ and $(\widetilde{A_j})_{j=1}^{\ell}$ be the cluster variables of the seed $\overrightarrow{\mu_{\underline{w}}}(\seed(\underline{w}))$ on $\cO^{\underline{w}}$.
		Then as rational functions on $\cO^\gamma$, we have
		\begin{equation*}
			\rho^*(A_j^+)=\begin{cases}
				\widetilde{A_{j+1}} &\text{if } i_{\ell-j+1} \neq i_\ell,\\
				{\widetilde{A_{j+1}}}\cdot {z_\ell^{-1}} &\text{if } i_{\ell-j+1}=i_\ell.
			\end{cases}
		\end{equation*}
	\end{lem}
	\begin{proof}
		We use the expression of the cluster variables in \ref{eqn:reversedseedvariables}. For $j>1$, we have
		\begin{align*}
			\widetilde{A_j}&=\Delta_{\omega_{i_{\ell-j+1}},\omega_{i_{\ell-j+1}}} (B_{\underline{w}[\ell-j+1,\ell]}(z_{\ell-j+1},\dots,z_\ell))\\
			&=\Delta_{\omega_{i_{\ell-j+1}},\omega_{i_{\ell-j+1}}} (B_{\underline{w}[\ell-j+1,\ell-1]}(z_{\ell-j+1},\dots,z_{\ell-1}) B_{i_\ell}(z_\ell))\\
			&=\Delta_{\omega_{i_{\ell-j+1}},\omega_{i_{\ell-j+1}}} (B_{\underline{w}[\ell-j+1,\ell-1]}(z_{\ell-j+1},\dots,z_{\ell-1}) u_{-i_\ell}(z_\ell^{-1}) \alpha_{i_\ell}^\vee(z_\ell) u_{i_\ell}(-z_{\ell}^{-1}))\\
			&=\Delta_{\omega_{i_{\ell-j+1}},\omega_{i_{\ell-j+1}}} (B_{\underline{w}[\ell-j+1,\ell-1]}(z'_{\ell-j+1},\dots,z'_{\ell-1}) \alpha_{i_\ell}^\vee(z_\ell))
		\end{align*}
		where the $z'_i$ arise bringing $u_{-i_\ell}(z_\ell^{-1})$ leftward. This is the definition of the map $\rho$, hence
		\begin{align*}
			\widetilde{A_j}&= \rho^*(A_{j-1}^+) \cdot \Delta_{\omega_{i_{\ell-j+1}},\omega_{i_{\ell-j+1}}}(\alpha_{i_\ell}^\vee(z_\ell)).\qedhere
		\end{align*}
	\end{proof}

	\begin{prop}\label{prop:pullbackseed}
		The freezing of the seed $\overrightarrow{\mu_{\underline{w}}} (\seed(\underline{w}))$ on $\cO^{\underline{w}}$ at the first cluster variable $\widetilde{A_1}=z_\ell$ is quasi-equivalent to the seed 
		on $\cO^\gamma = \cO^{(i_1,\dots, i_{\ell-1},-i_\ell)}$ given by
		$$\overrightarrow{\mu_{\underline{w}[\ell-1]}}(\seed(\gamma)) =  \rho^*(\overrightarrow{\mu_{\underline{w}[\ell-1]}}(\seed(\underline{w}[\ell-1]))) \sqcup \left\{\boxed{z_\ell'}\right\}.$$
	\end{prop}
	\begin{proof}
		Our seed $\seed(\gamma)$ for $\cO^\gamma$ has one extra frozen variable $z_\ell'=\frac{1}{z_\ell}$. By Lemma \ref{lem:pullbackvars}, we can rescale the variables $\rho^*(A_j^+)$ by a monomial in this frozen variable to get the desired cluster variables $\widetilde{A_j}$. It remains to check that doing so preserves the exchange ratios at each mutable vertex $k$.
		From Lemma \ref{lem:pullbackvars}, the degree of $z_\ell$ in the exchange ratio $y_k$ will be
		$$D_k = -\sum_j \ep_{jk} \delta_{i_{\ell-j-1},i_\ell}, \quad \text{where}\quad \delta_{i_{\ell-j-1},i_\ell} = \begin{cases}
			1 &\text{if } i_{\ell-j-1} = i_\ell,\\
			0 &\text{if } i_{\ell-j-1} \neq i_\ell.
		\end{cases}$$
		Let $s$ and $p$ denote the successor and predecessor functions of the reversed word $(i_\ell,\dots,i_1)$.
		First consider the case where $i_k\neq i_\ell$. From the structure of the exchange matrix $E(\overline{i_{\ell-1}},\dots,\overline{i_1})$,
		exactly one of the following occurs for the column $(\ep_{1,k},\dots,\ep_{\ell-1,k})^T$:
		\begin{enumerate}[(i)]
			\item The column is such that $\ep_{j,k}=0$ whenever $i_j=i_\ell$, and in this case $D_k=0$.
			\item There are exactly two indices $j_1<k<j_2$ with $i_{j_1} = i_{j_2} = i_\ell$ with $\ep_{j_1,k},\ep_{j_2,k}\neq0$. We have $j_1<k<s(j_1)$, $k<j_2<s(k)$, so $\ep_{j_1,k} = a_{i_k,i_\ell}=-\ep_{j_2,k}$, and thus $D_k=0$.
			
			\item There is exactly one index $j$ with $i_{j}= i_\ell$ and $\ep_{j,k}\neq0$. In this case, $k<j<s(k)$, $\ep_{j,k} = -a_{i_k,i_\ell}$ and $D_k = a_{i_k,i_\ell}$.
		\end{enumerate}
		A similar consideration when $i_k=i_\ell$ shows that for mutable vertices $(i_\ell,m)$ with $m>1$ on level $i_\ell$, they are all have $\ep_{p(k),k} = 1 = -\ep_{s(k),k}$, and therefore $D_k=0$. But there is at most one vertex, namely $k=(i_{\ell},1)$ the first vertex on level $i_\ell$, which has no predecessor. For this index we have $D_k=1$.
		
		Therefore, the degree $D_k$ is zero at any mutable vertex $k$ except for those indices in case (iii) and the first index in the level $i_\ell$.
		To preserve exchange ratios when rescaling the factor of $z_\ell$ in the cluster variables, we add $D_k$ arrows from the frozen vertex $\boxed{z_\ell}$ to $k$.
		This recovers the freezing of $z_\ell$ in the seed $\overrightarrow{\mu_{\underline{w}}} (\seed(\underline{w})) = \seed(\overline{i_\ell},\overline{i_{\ell-1}},\dots,\overline{i_2},\overline{i_1})$.
	\end{proof}
	We also denote by $\overleftarrow{\mu_{\underline{w}}}$ the reversed sequence of mutations, so that $\overleftarrow{\mu_{\underline{w}}}\circ \overrightarrow{\mu_{\underline{w}}} (\seed(\underline{w})) = \seed(\underline{w})$.

	\begin{prop}\label{prop:pullbackcompatiblewithchangeofcoords}
		For $\gamma = \underline{w}(\ell) = (i_1,i_2,\dots,i_{\ell-1},-i_\ell)$, there exists a rational quasi-cluster morphism $\cO^{\underline{w}}\to \cO^{\underline{w}[\ell-1]}$ such that the following picture is commutative:
		\[\begin{tikzcd}[row sep = small]
			{\left(\cO^{\underline{w}}, \overleftarrow{\mu_{\underline{w}[\ell-1]}}\circ \overrightarrow{\mu_{\underline{w}}} (\seed(\underline{w}))\right) } && {\cO^{\gamma}} & {\cO^{\underline{w}[\ell-1]} \times \mathbb{A}^1} \\
			\\
			&& {(\cO^{\underline{w}[\ell-1]},\seed(\underline{w}[\ell-1]))}
			\arrow[dashed, tail reversed, from=1-1, to=1-3]
			\arrow["{{\bigl(J(\underline{w}[\ell-1]) { , } \{(i_\ell, 1)\}\bigr)}}"', dashed, from=1-1, to=3-3]
			\arrow["\cong", tail reversed, from=1-3, to=1-4]
			\arrow["{{\rho}}", from=1-3, to=3-3]
			\arrow["{\rho_1}", from=1-4, to=3-3]
		\end{tikzcd}\]
	\end{prop}
	\begin{proof}
		This follows from Lemma \ref{lem:pullbackvars}, Proposition \ref{prop:pullbackseed}, and applying the reversing sequence $\overleftarrow{\mu_{\underline{w}[\ell-1]}}$ to both seeds. 
		Note that there is a change of the indexing in the morphism. The embedding of seed vertices $\overrightarrow{\mu_{\underline{w}[\ell-1]}}(\seed(\underline{w}[\ell-1])) \to \overrightarrow{\mu_{\underline{w}}}(\seed(\underline{w}))$ is given by $(i,a)\mapsto(i,a)$ if $i\neq i_\ell$ and $(i,a)\mapsto(i,a+1)$ if $i=i_\ell$, so the second mutation sequence factor $\overleftarrow{\mu_{\underline{w}[\ell-1]}}$ is applied to vertices in level $i_\ell$ shifted by one position.
	\end{proof}

	\begin{eg}
		Continuing Example \ref{eg:(1,2,1,2,1,2)} and applying $\overleftarrow{\mu_{\underline{w}[\ell-1]}} = (\mu_1\circ\mu_3)\circ\mu_4\circ\mu_1\circ\mathrm{id}\circ\mathrm{id}$, the seed $\overleftarrow{\mu_{\underline{w}[\ell-1]}} \circ \overrightarrow{\mu_{\underline{w}}}(\seed(\underline{w}))$ has quiver
		\[\begin{tikzcd}[sep=small]
			& 2 && 4 && {\boxed{6}} \\
			1 && 3 && {\boxed{5}}
			\arrow[from=1-2, to=2-3]
			\arrow[from=1-4, to=2-3]
			\arrow[from=1-6, to=1-4]
			\arrow[dashed, from=1-6, to=2-5]
			\arrow[from=2-1, to=1-2]
			\arrow[from=2-1, to=1-4]
			\arrow[from=2-3, to=1-6]
			\arrow[from=2-3, to=2-1]
			\arrow[from=2-5, to=2-3]
		\end{tikzcd}\quad \sim \quad
		\begin{tikzcd}[sep=small]
			2 && 4 && {\boxed{6}} & \\
			& 1 && 3 && {\boxed{5}}
			\arrow[from=1-1, to=2-4]
			\arrow[from=1-3, to=2-4]
			\arrow[from=1-5, to=1-3]
			\arrow[dashed, from=1-5, to=2-6]
			\arrow[from=2-2, to=1-1]
			\arrow[from=2-2, to=1-3]
			\arrow[from=2-4, to=1-5]
			\arrow[from=2-4, to=2-2]
			\arrow[from=2-6, to=2-4]
		\end{tikzcd}\]
		and mutable variables $A_2=z_6$, $A_1 = z_1$, $A_4=z_2$, and $A_3 = (z_1z_3-z_2)z_6-z_1$. Note the copy of the quiver for $\seed(\gamma^+)$ supported on the vertices $J\setminus\{2\} = J\setminus \{(2,1)\}$.
	\end{eg}
	
	If $\underline{v} = (i_1,\dots,i_m)$ is a prefix of $\underline{w} = (i_1,\dots,i_\ell)$ with $m<\ell$, there is an embedding of seeds $\seed(\underline{v})\hookrightarrow \seed(\underline{w})$. Applying $\overrightarrow{\mu_{\underline{v}}}$ to $\seed(\underline{w})$, we get the seed $\seed(\overline{i_m}, \overline{i_{m-1}},\dots, \overline{i_2},\overline{i_1},i_{m+1},\dots,i_\ell)$, with exchange matrix defined as in \citep[\S2.7]{BY25} and whose cluster variables are
	\begin{equation}\label{eqn:reorderedseedvariables}
		\widetilde{A_j} = \begin{cases}
			\Delta_{\omega_{i_{m-j+1}},\omega_{i_{m-j+1}}} (B_{\underline{w}[m-j+1,m]}(z_{m-j+1},\dots,z_m)) & \text{ if } 1\leq j\leq m,\\
			\Delta_{\omega_{i_j},\omega_{i_j}} (B_{\underline{w}[j]}(z_1,\dots,z_j)) & \text{ if } m<j\leq \ell.
		\end{cases}
	\end{equation}

	\begin{lem}\label{lem:pullbackvars2}
		Let $\gamma = \underline{w}(m)$, be the subexpression obtained by negating the $m$-th letter.
		Let $(A_j^+)_{j=1}^{\ell-1}$ be the cluster variables of the seed $\overrightarrow{\mu_{\underline{w}[m-1]}}(\seed(\underline{w}[m-1],\underline{w}[m+1,\ell]))$ on $\cO^{\gamma^+}$ and $(\widetilde{A_j})_{j=1}^{\ell}$ be the cluster variables of the seed $\overrightarrow{\mu_{\underline{w}}[m]}(\seed(\underline{w}))$ on $\cO^{\underline{w}}$.
		Then as rational functions on $\cO^\gamma$, under the pullback by the map $\rho : \cO^\gamma \to \cO^{\gamma^+}$, we have
		\begin{equation*}
			\rho^*(A_j^+)=\begin{cases}
				\widetilde{A_{j+1}} &\text{if } i_{m-j+1} \neq i_m \text{ and } 1\leq j<m,\\
				\widetilde{A_{j+1}}\cdot z_m^{-1} &\text{if } i_{m-j+1}=i_m  \text{ and } 1\leq j<m,\\
				\widetilde{A_{j+1}}\cdot  z_m^{-\langle\underline{w}^{j-1}(\omega_{i_j}),\underline{w}^{m+1}(\alpha_{i_m}^\vee)\rangle} &\text{if } m\leq j \leq \ell-1,
			\end{cases}
		\end{equation*}
		\begin{proof}
			Similar to the proof of Lemma \ref{lem:pullbackvars}, we use the expression of the cluster variables in \ref{eqn:reorderedseedvariables} and apply identities corresponding to fraying a strand.
			The first two cases are the same as in the proof of Lemma \ref{lem:pullbackvars}. For $m+1\leq j\leq \ell$, we have
			\begin{align*}
				\widetilde{A_j}&=\Delta_{\omega_{i_{j}},\omega_{i_{j}}} (B_{\underline{w}[j]}(z_{\ell-j+1},\dots,z_\ell))\\
				&=\Delta_{\omega_{i_{j}},\omega_{i_{j}}} (B_{\underline{w}[m-1]}(z_{1},\dots,z_{m-1}) B_{i_m}(z_m)B_{\underline{w}[m+1,j]}(z_{m+1},\dots,z_j))\\
				&=\Delta_{\omega_{i_{j}},\omega_{i_{j}}} (B_{\underline{w}[m-1]}(z_{1},\dots,z_{m-1}) u_{-i_m}(z_m^{-1}) \alpha_{i_m}^\vee(z_m) u_{i_m}(-z_{m}^{-1})B_{\underline{w}[m+1,j]}(z_{m+1},\dots,z_j))\\
				&=\Delta_{\omega_{i_{j}},\omega_{i_{j}}} (B_{\underline{w}[m-1]}(z'_{1},\dots,z'_{m-1})B_{\underline{w}[m+1,j]}(z_{m+1}',\dots,z_j') ((s_{i_j}\cdots s_{i_{m+1}})\cdot \alpha_{i_m})^\vee(z_m) \cdot U)
			\end{align*}
			where the $z'_i$ arise from the coordinate change from bringing $u_{-i_\ell}(z_\ell^{-1})$ to the left and the Borel element $\alpha_{i_m}^\vee(z_m) u_{i_m}(-z_{m}^{-1})$ to the right. The first factor here matches the construction of the map $\rho$, and therefore
			\begin{align*}
				\widetilde{A_j}&= \rho^*({A_{j-1}^+}) \cdot \Delta_{\omega_{i_{j}},\omega_{i_{j}}}(((s_{i_j}\cdots s_{i_{m+1}})\cdot \alpha_{i_m})^\vee(z_m))\\
				&= \rho^*({A_{j-1}^+}) \cdot z_m^{\langle\omega_{i_j},(s_{i_j}\cdots s_{i_{m+1}})(\alpha_{i_m}^\vee)\rangle}. \qedhere
			\end{align*}
		\end{proof}
	\end{lem}
	
	\begin{prop}\label{prop:pullbackseed2}
		For $1\leq m\leq \ell$ and $\gamma = \underline{w}(m)$, the freezing of the seed $\overrightarrow{\mu_{\underline{w}[m]}} (\seed(\underline{w}))$ at the first cluster variable $\widetilde{A_1}=z_m$ is quasi-equivalent to the seed on $\cO^\gamma$ given by
		$$\overrightarrow{\mu_{\underline{w}[m-1]}}(\seed(\gamma)) =  \rho^*(\overrightarrow{\mu_{\underline{w}[m-1]}}(\seed(\gamma^+))) \sqcup \left\{\boxed{z_m'}\right\}.$$
	\end{prop}
	\begin{proof}
		The proof is an extension of that from Proposition \ref{prop:pullbackseed}. As before, we wish to rescale the cluster variables by appropriate powers of the frozen variable $z_m' = \frac{1}{z_m}$. We therefore need to check the degree $D_k$ of $z_m$ in the exchange ratio $y_k$ in $\rho^*(\overrightarrow{\mu_{\underline{w}[m-1]}}(\seed(\gamma^+))$.
		The case for indices with $k<m$ is checked in the proof of Proposition \ref{prop:pullbackseed}, and we now determine $D_k$ for mutable indices $k> m$. For $j>m$, let $\beta_j$ denote the integer $-\langle\underline{w}^{j-1}(\omega_{i_j}),\underline{w}^{m+1}(\alpha_{i_m}^\vee)\rangle = -\langle \omega_{i_j}, (s_{i_j}\cdots s_{i_{m+1}})(\alpha_{i_{m}}^\vee)\rangle$, and for $j<m$, let $\beta_j = -\delta_{i_{m-j-1},i_m}$ as in the proof of Proposition \ref{prop:pullbackseed}.
		By the action of simple reflections,
		$$\beta_{s(k)} = -\beta_k - \left(\sum_{j<s(k)<s(j)} a_{i_k,i_j}\beta_j\right) \; \text{ and } \; -\beta_{p(k)} =  \beta_{k} + \left(\sum_{j<k<s(j)} a_{i_k,i_j}\beta_j\right).$$

		Then for $k>m$ with one of $k>s(m)$ or $s(k)<s(m)$, we have $D_k= \sum_{j}\ep_{jk}\beta_{j}$. So
		\begin{align*}
			D_k
			&= \beta_{s(k)} + \left( \sum_{ k<j<s(k)<s(j)}a_{i_k,i_j}\beta_{j} \right) -\beta_{p(k)} -\left( \sum_{j<k<s(j)<s(k)}a_{i_k,i_j}\beta_{j}\right)\\
			&= -\beta_k - \left( \sum_{ j<k<s(k)<s(j)}a_{i_k,i_j}\beta_{j}\right) +\beta_k + \left( \sum_{j<k<s(k)<s(j)}a_{i_k,i_j}\beta_{j}\right)= 0.
		\end{align*}
		So as in the proof of Proposition \ref{prop:pullbackseed}, a nonzero degree $D_k$ occurs only at positions $k$ with $i_k= i_m$  when $m<k<s(m)<s(k)$ or $k=s(m)$. In the case $k=s(m)$, so  position $k$ is the first instance of $i_m$ in the word $(\overline{i_{m-1}},\dots, \overline{i_2},\overline{i_1},i_{m+1},\dots,i_\ell)$, we have
		\begin{align*}
			D_k
			&= -\beta_k - \left( \sum_{ j<k<s(k)<s(j)}a_{i_k,i_j}\beta_{j}\right) -\left( \sum_{j<k<s(j)<s(k)}a_{i_k,i_j}\beta_{j}\right)= -\langle\omega_{i_m},\alpha_{i_m}^\vee\rangle=-1.
		\end{align*}
		When $m<k<s(m)<s(k)$, we have $i_k\neq i_m$, and so
		\begin{align*}
			D_k
			&= -\beta_k - \left( \sum_{ j<k<s(k)<s(j)}a_{i_k,i_j}\beta_{j}\right) - \beta_{p(k)}-\left( \sum_{\substack{j<k<s(j)<s(k)\\j\neq m}}a_{i_k,i_j}\beta_{j}\right)= -a_{i_k,i_m},
		\end{align*}
		Adding $D_k$ many arrows from the new frozen vertex to vertex $k$ therefore preserves the exchange ratios when rescaling to remove all the factors of $z_m$ from the cluster variables.
	\end{proof}
	
	\begin{rem}
		The above arguments were inspired by a definition of \emph{regular pullback of seeds} from \citep[\S2.3]{GSV26a} and \citep{GSV26b}. The procedure here is slightly different as we rescale the variables by all powers of the poles, not just those factors appearing in the denominators.
	\end{rem}

	\begin{cor}\label{prop:pullbackcompatiblewithchangeofcoords2}
		For $\gamma = \underline{w}(m) = (i_1,\dots,i_{m-1},-i_m,i_{m+1},\dots,i_\ell)$, we similarly have:
		\[\begin{tikzcd}[column sep=small, row sep = small]
			{\left(\cO^{\underline{w}}, \overleftarrow{\mu_{\underline{w}[m-1]}}\circ \overrightarrow{\mu_{\underline{w}[m]}} (\seed(\underline{w}))\right) } && {\cO^{\gamma}} & {\cO^{\gamma^+} \times \mathbb{A}^1} \\
			\\
			&& {(\cO^{\gamma^+},\seed(\underline{w}[m-1],\underline{w}[m+1,\ell])).}
			\arrow[dashed, tail reversed, from=1-1, to=1-3]
			\arrow["{{\bigl(J(\underline{w}[m-1],\underline{w}[m+1,\ell]) { , } \{(i_m, 1)\}\bigr)}}"', dashed, from=1-1, to=3-3]
			\arrow["\cong", from=1-3, to=1-4]
			\arrow["\rho", from=1-3, to=3-3]
			\arrow["{{\rho_1}}", from=1-4, to=3-3]
		\end{tikzcd}\]
	\end{cor}
	\begin{proof}
		The proof is identical to that of Proposition \ref{prop:pullbackcompatiblewithchangeofcoords} and follows analagously from Lemma \ref{lem:pullbackvars2}, Proposition \ref{prop:pullbackseed2}, and applying the reversing sequence $\overleftarrow{\mu_{\underline{w}[m-1]}}$ to both seeds. As before, there is a change of the indexing where the second mutation sequence factor $\overleftarrow{\mu_{\underline{w}[m-1]}}$ is applied to vertices in level $i_m$ shifted by one position.
	\end{proof}

	\begin{eg}\label{eg:mutationsequences(1,2,1,-2)}
		Continuing with Example \ref{eg:(1,2,1,-2)}, let $\underline{w}=(1,2,1,2)$, $\gamma=(1,2,1,-2)$, and $\gamma^+ = \underline{w}[\ell-1] = (1,2,1)$ in type $A$.
		Denote the Bott--Samelson coordinates on the charts $\cO^{\underline{w}}$, $\cO^{\gamma}$, and $\cO^{\gamma^+}$ by $z_i$, $z_i'$, and  $\widetilde{z}_i$ respectively.
		The seeds $\seed(\underline{w})$ and $\seed(\gamma^+)$ are respectively:
		\[\begin{tikzcd}[sep=small]
			& {z_2} && {\boxed{z_2z_4-z_3}} \\
			{z_1} && {\boxed{z_1z_3-z_2}}
			\arrow[from=1-2, to=2-3]
			\arrow[from=1-4, to=1-2]
			\arrow[from=2-1, to=1-2]
			\arrow[dashed, from=2-3, to=1-4]
			\arrow[from=2-3, to=2-1]
		\end{tikzcd}\quad \text{and} \quad \begin{tikzcd}[sep=tiny]
			& {\boxed{\widetilde{z}_2}} & \\
			{\widetilde{z}_1} && {\boxed{\widetilde{z}_1\widetilde{z}_3-\widetilde{z}_2}}
			\arrow[dashed, from=1-2, to=2-3]
			\arrow[from=2-1, to=1-2]
			\arrow[from=2-3, to=2-1]
		\end{tikzcd}\]		
		We can see that the truncation $\psi:\cO^{\underline{w}} \to \cO^{\underline{w}[\ell-1]}$, $\psi^*(\widetilde{z}_i)=z_i$ is rational quasi-cluster.
		Applying the mutation sequence $\overleftarrow{\mu_{\underline{w}[\ell-1]}}\circ \overrightarrow{\mu_{\underline{w}}} =  \mu_1\circ(\mu_2\circ\mu_1)$ to $\seed(\underline{w})$, we get the seed
		\[\begin{tikzcd}[sep=small]
			& {z_4} && {\boxed{z_2z_4-z_3}} \\
			{z_1z_4-1} && {\boxed{z_1z_3-z_2}}
			\arrow[from=1-2, to=2-1]
			\arrow[dashed, from=1-4, to=2-3]
			\arrow[from=2-1, to=1-4]
			\arrow[from=2-3, to=2-1]
		\end{tikzcd}\]
		on $\cO^{(1,2,1,2)}$ where we may freeze and isolate the coordinate $z_4$ as one of our cluster variables.
		The Poisson map described in Proposition \ref{prop:mapstosupportedsubword} (see also Figure \ref{fig:frayedweaveequivlabelled}) is given by
		$$\rho: \cO^{(1,2,1-2)} \to \cO^{(1,2,1)}, \quad \Phi_{\gamma}(z'_1,z'_2,z'_3,z'_4) \mapsto  \Phi_{\gamma^+}(z'_1-z'_4,z'_2-z'_3z'_4,z'_3).$$
		Pulling back $\seed(\gamma^+)$ via $\rho$ and adding an isolated frozen $z_4'$, the initial seed $\seed(\gamma)$ is:
		\[\begin{tikzcd}[sep=small]
			{\boxed{z'_4}} && {\boxed{z'_2-z'_3z'_4}} & \\
			& {z'_1-z'_4} && {\boxed{z'_1z'_3-z'_2}}
			\arrow[dashed, from=1-3, to=2-4]
			\arrow[from=2-2, to=1-3]
			\arrow[from=2-4, to=2-2]
		\end{tikzcd}\]
		Using the coordinate change $z_4'=z_4^{-1}$, observe that the exchange ratios for the bottom left variable are preserved between the initial seed of $\cO^{(1,2,1,-2)}$ and that of $\mu_{1}\circ\mu_2\circ\mu_1(\seed(\underline{w}))$, and the variables are rescaled by the isolated frozen $z_4' = \tfrac{z_1'-z_4'}{z_1z_4-1}$.
		The change of coordinates $\cO^{(1,2,1,2)} \dashrightarrow\cO^{(1,2,1,-2)}$ is rational quasi-cluster and composes with $\rho$ to give a rational quasi-cluster map $\cO^{(1,2,1,2)} \dashrightarrow\cO^{(1,2,1)}$.
	\end{eg}
	
	\subsection{The mutation sequence $\overrightarrow{\mu_\gamma}$}
	Let $\gamma$ be any subexpression of $\underline{w}$ and let $j_1<\dots<j_m$ be the indices where $\gamma_{j_l}<0$. Denote by $\gamma_{(l)}$ is the subexpression of $\underline{w}$ with $\gamma_{j_{l'}}<0$ for $l'\leq l$, and $\gamma_k>0$ else.	
	The transition map from $\cO^{\underline{w}}$ to $\cO^{\gamma}$ is a composition
	$$\cO^{\underline{w}} \dashrightarrow \cO^{\gamma_{(1)}}\dashrightarrow \cO^{\gamma_{(2)}} \dashrightarrow\cdots \dashrightarrow \cO^{\gamma_{(m)}}=\cO^{\gamma}$$
	of maps between adjacent charts. Corollary \ref{prop:pullbackcompatiblewithchangeofcoords2} then motivates the following:
	
	\begin{defn}\label{defn:rhomutationsequence}
		We iteratively define $\overrightarrow{\mu_{\gamma_{(l)}}} =  (\overleftarrow{\mu_{\gamma_{(l-1)}^+[j_l-1]}}\circ\overrightarrow{\mu_{\gamma_{(l-1)}^+[j_l]}}) \circ \overrightarrow{\mu_{\gamma_{(l-1)}}}$,
		so that
		\begin{equation*}
			\overrightarrow{\mu_{\gamma}} := (\overleftarrow{\mu_{\gamma_{(m-1)}^+[j_m-m]}}\circ\overrightarrow{\mu_{\gamma_{(m-1)}^+[j_m-(m-1)]}})\circ \cdots \circ (\overleftarrow{\mu_{\gamma_{(1)}^+[j_2-2]}}\circ\overrightarrow{\mu_{\gamma_{(1)}^+[j_2-1]}})\circ (\overleftarrow{\mu_{\underline{w}[j_1-1]}}\circ\overrightarrow{\mu_{\underline{w}[j_1]}}).
		\end{equation*}
	\end{defn}
	
	\begin{eg}
		Continuing Example \ref{eg:(2,3,1,-2,2,-1)}, let $\underline{w} = (2,3,1,2,2,1)$, so that $\seed(\underline{w})$ for $\cO^{\underline{w}}$ is:
		\[\begin{tikzcd}[sep=tiny]
			& {\boxed{z_2}} &&&&&& \\
			{z_1} &&& {z_1z_4-z_2z_3} &&& {\boxed{(z_1z_4-z_2z_3)z_5-z_1}} \\
			&& {z_3} &&&&& {\boxed{z_3z_6-z_4z_5+1}}
			\arrow[dashed, from=1-2, to=2-7]
			\arrow[from=2-1, to=1-2]
			\arrow[from=2-1, to=3-3]
			\arrow[from=2-4, to=2-1]
			\arrow[from=2-7, to=2-4]
			\arrow[dashed, from=2-7, to=3-8]
			\arrow[from=3-3, to=2-7]
			\arrow[from=3-8, to=3-3]
		\end{tikzcd}\]
		When $\gamma = (2,3,1,-2,2,-1)$, applying the mutation sequence $\overrightarrow{\mu_{\gamma}}$ gives the seed
		\begin{align*}
			\overrightarrow{\mu_\gamma}(\seed(\underline{w})) &= (\overleftarrow{\mu_{\gamma_{(1)}^+[5]}}\circ\overrightarrow{\mu_{\gamma_{(1)}^+[5]}})\circ (\overleftarrow{\mu_{\underline{w}[3]}}\circ\overrightarrow{\mu_{\underline{w}[4]}})(\seed(\underline{w}))=\bigl(\mu_4\circ \mu_3\circ \mu_4\bigr) \circ \bigl(\mathrm{id}\circ\mu_1\bigr)(\seed(\underline{w})):
		\end{align*}
		\[\begin{tikzcd}[sep=tiny]
			& {\boxed{z_2}} &&&&&& \\
			{{z_4}} &&& {(z_1z_4-z_2z_3)z_6-z_2} &&& {\boxed{(z_1z_4-z_2z_3)z_5-z_1}} \\
			&& {{z_6}} &&&&& {\boxed{z_3z_6-z_4z_5+1}}
			\arrow[from=1-2, to=2-1]
			\arrow[dashed, from=1-2, to=2-7]
			\arrow[from=1-2, to=3-3]
			\arrow[from=2-1, to=2-4]
			\arrow[from=2-4, to=1-2]
			\arrow[curve={height=6pt}, from=2-4, to=3-8]
			\arrow[from=2-7, to=2-4]
			\arrow[from=3-3, to=2-4]
			\arrow[curve={height=-12pt}, from=3-8, to=2-1]
			\arrow[dashed, from=3-8, to=2-7]
		\end{tikzcd}\]
		Freezing and rescaling by $z_4$ and $z_6$ gives a quasi-equivalence to $\seed(\gamma)$ from Example \ref{eg:(2,3,1,-2,2,-1)}.
	\end{eg}

	\begin{thm}
		The seed $\overrightarrow{\mu_{\gamma}} (\seed(\underline{w}))$ is an iterated pullback and defrosting of the seed $\seed(\gamma^+)$ under the Poisson rational maps
		$\cO^{\underline{w}} \dashrightarrow \cO^{\gamma_{(1)}^+}\dashrightarrow \cO^{\gamma_{(2)}^+} \dashrightarrow\cdots \dashrightarrow  \cO^{\gamma^+}$.
		The embeddings of seeds $\seed(\gamma^+) \hookrightarrow \seed(\gamma) \rightarrow \overrightarrow{\mu_{\gamma}} (\seed(\underline{w}))$ gives the diagram of rational quasi-cluster morphisms:
		\[\begin{tikzcd}[row sep= small]
			{\left(\cO^{\underline{w}}, \overrightarrow{\mu_{\gamma}} (\seed(\underline{w}))\right) } && {\cO^{\gamma}} & {\cO^{\gamma^+} \times \mathbb{A}^m} \\
			\\
			&& {(\cO^{\gamma^+},\seed(\gamma^+))}
			\arrow[dashed, tail reversed, from=1-1, to=1-3]
			\arrow["{{\bigl(J(\gamma^+) { , } S_{\gamma}\bigr)}}"', dashed, from=1-1, to=3-3]
			\arrow["\cong", from=1-3, to=1-4]
			\arrow["\rho", from=1-3, to=3-3]
			\arrow["{{\rho_1}}", from=1-4, to=3-3]
		\end{tikzcd}\]
	\end{thm}
	\begin{proof}
		The first part follows from the definition $\overrightarrow{\mu_{\gamma}}$ and applying Proposition \ref{prop:pullbackseed2} to each composition factor.
		The second part is the iterative application of Corollary \ref{prop:pullbackcompatiblewithchangeofcoords2}.
	\end{proof}
	
	\begin{rem}
		Our construction of $\rho:\cO^\gamma \to \cO^{\gamma^+}$ relied on the longest element $w_0$ and therefore does not extend to the Kac--Moody setting of \citep{BY25}. We still expect a Poisson isomorphism $\cO^\gamma \cong \cO^{\gamma^+}\times \mathbb{A}^m$, just not with this construction. The quasi-equivalence arguments just use the Cartan matrix and should extend to the Kac--Moody setting.
	\end{rem}
	
	\begin{rem}
		Different orderings of the unsupported letters of $\gamma$ give different compositions of maps between adjacent charts, and may yield different intermediary poles.
		What isolated frozens variables are added to $\rho^*(\seed(\gamma^+))$ to get a seed on $\cO^{\underline{\gamma}}$, therefore depends on the order in which we negate the letters.
		This is also reflected in how corresponding mutation sequences need not be equal. The next subsection discusses how this left-to-right choice is already present in the literature, but it might be interesting to study seeds which arise from other orderings.
	\end{rem}
	\begin{eg}
		When
		$\gamma=(1,-1,-1)$.
		the composition $\cO^{(\underline{w})}\dasharrow\cO^{\gamma} \to \cO^{(1)}$ is given by
		$$\Phi_{\underline{w}}(z_1,z_2,z_3) \mapsto \Phi_{(1)}\left(\frac{z_1z_2z_3-z_1z_3}{z_2z_3-1}\right).$$
		In particular, the rational map only has one irreducible component of the denominators and the regular pullback procedure from \citep{GSV26a,GSV26b} would only add one additional vertex. Choosing a path of adjacent charts adds to the indeterminacy locus, and allows us to use the previous results to construct seeds for $\underline{w}=(1,1,1)$ from those of $\cO^{\gamma}$ and $\cO^{(1)}$.
		Note however, different orderings of adjacent charts can yield different seeds.
		The seed $\seed(\underline{w})$ is:
		\[\begin{tikzcd}
			{z_1} & {z_1z_2-1} & {\boxed{z_1z_2z_3-z_1-z_3}}
			\arrow[from=1-2, to=1-1]
			\arrow[from=1-3, to=1-2]
		\end{tikzcd}\]
		Let $\gamma' = (1,-1,1)$ and $\gamma''=(1,1,-1)$. Respectively applying the  two sequences $\overrightarrow{\mu_{\gamma'}}=\overleftarrow{\mu_{\underline{w}[1]}}\circ\overrightarrow{\mu_{\underline{w}[2]}} = \mathrm{id} \circ\mu_1$ and $\overrightarrow{\mu_{\gamma''}}=\overleftarrow{\mu_{\underline{w}[2]}}\circ\overrightarrow{\mu_{\underline{w}}} = \mu_2\circ \mu_1 \circ(\mu_2\circ\mu_1)$ to $\seed(\underline{w})$ yields the seeds: 
		\[\begin{tikzcd}[sep=small]
			{z_2} & {z_1z_2-1} & {\boxed{z_1z_2z_3-z_1-z_3}}
			\arrow[from=1-1, to=1-2]
			\arrow[from=1-3, to=1-2]
		\end{tikzcd} \quad\text{ and }\quad
		\begin{tikzcd}[sep=small]
			{z_3} & {z_1} & {\boxed{z_1z_2z_3-z_1-z_3}}
			\arrow[curve={height=30pt}, from=1-1, to=1-3]
			\arrow[from=1-2, to=1-1]
			\arrow[from=1-3, to=1-2]
		\end{tikzcd}
		\]
		Composing both with $\overleftarrow{\mu_{(\gamma)^+[1]}}\circ\overrightarrow{\mu_{(\gamma')^+}} =  \overleftarrow{\mu_{(\gamma)^+[1]}}\circ\overrightarrow{\mu_{(\gamma'')^+}}=\mathrm{id} \circ \mu_2$ gives the seeds:
		\[\begin{tikzcd}[sep=small]
			{z_2} & {z_2z_3-1} & {\boxed{z_1z_2z_3-z_1-z_3}}
			\arrow[from=1-2, to=1-1]
			\arrow[from=1-2, to=1-3]
		\end{tikzcd} \;\text{ and }\;
		\begin{tikzcd}[sep=small]
			{z_3} & {z_2z_3-1} & {\boxed{z_1z_2z_3-z_1-z_3}}
			\arrow[from=1-1, to=1-2]
			\arrow[from=1-2, to=1-3]
		\end{tikzcd}\]
	\end{eg}

	\subsection{Comparison with M\'enard's mutation sequence}
	In \citep{M22}, M\'enard constructed a specific of mutations, freezings, and deletions which takes $\seed(\underline{w})$ to a seed for open Richardson varieties. Bao and Ye \citep{BY25} generalised this algorithm for Kac--Moody open Richardson varieties $\mathring{Z}_{\underline{w}}^v$ in the twisted product of flags. Here, we show the relationship between the sequence $\overrightarrow{\mu_{\gamma}}$ defined in the previous section and M\'enard's  mutation sequence $M(\gamma)$.
	\begin{defn}[\protect{Adapted from \citep[Definition 6.1]{M22} and \citep[Definition 5.1]{BY25}}]
		As before, say the letter $i$ appears $n_i$ times in $\underline{w}$ and for $1\leq j\leq \ell$ let $i_j=i$ be the $a_j$-th appearance of $i=i_j$ in $\underline{w}$. Further define $b_j = |\{ k < j\mid \gamma_k = -i\}|$ to be the number of instances of $i$ before position $j$ that are \emph{not} supported the subexpression.
		
		We define $M(\gamma) = \widetilde{\widetilde{\mu}}_{\ell} \circ \widetilde{\widetilde{\mu}}_{\ell-1}\circ \cdots \circ \widetilde{\widetilde{\mu}}_2\circ \widetilde{\widetilde{\mu}}_1$, where $\widetilde{\widetilde{\mu}}_j = \widetilde{\widetilde{\mu}}_j(\gamma)$ is given by
		\begin{equation}
			\widetilde{\widetilde{\mu}}_j = \begin{cases}
				\mu_{(i,n_i-(a_j-b_j))}\circ\cdots \circ \mu_{(i,b_j+1)} & \text{ if } a_l<n_i \text{ and } \gamma_j=i,\\
				\mathrm{id} & \text{ if } \gamma_j=-i,\\
				\mathrm{id} & \text{ if } a_j=n_i.
			\end{cases}
		\end{equation}
	\end{defn}
	As described in \citep{M22}, this sequence mutates in ascending order on each level for each supported letter, but skips some initial mutations for each unsupported letter and a number of final mutations dependent on how many times we have already mutated along level.
	\begin{rem}
		While the results for open Richardson varieties $\mathring{Z}^v_{\underline{w}}$ in \citep{M22,BY25} require $\gamma$ to be the left-most subexpression of $v=\gamma^+$, the sequence $M(\gamma)$ can be defined for any $\gamma \in \Upsilon_{\underline{w}}$.
	\end{rem}
	\begin{lem}\label{lem:Menardfactoring}
		Let $\gamma = (\gamma',\gamma_{\ell})$ be a subexpression of $\underline{w}=(\underline{w}',i)$, and $b=b_\ell$ be the number of instances of $-i$ in $\gamma'$.
		Then
		pushing forward M\'enard's sequence $M(\gamma')$ under the inclusion of seeds $\seed(\underline{w}') \hookrightarrow \seed(\underline{w})$,
		we have
		\begin{equation}\label{eqn:Menardfactoring}
			M(\gamma)(\seed(\underline{w})) = \left(\mu_{(i,b+1)}\circ\dots\circ\mu_{(i,n_i-1)}\right)\circ  M(\gamma')(\seed(\underline{w})).
		\end{equation}
		
	\end{lem}
	\begin{proof}
		Consider the expressions $M(\gamma)= \widetilde{\widetilde{\mu}}_{\ell}\circ \cdots \circ \widetilde{\widetilde{\mu}}_2\circ \widetilde{\widetilde{\mu}}_1$ and $M(\gamma')= \widetilde{\widetilde{\mu}}_{\ell-1}'\circ \cdots \circ \widetilde{\widetilde{\mu}}_2'\circ \widetilde{\widetilde{\mu}}_1'$.
		Note $\widetilde{\widetilde{\mu}}_\ell = \mathrm{id}$. For $j<\ell$, we have that $\widetilde{\widetilde{\mu}}_j=\mu_{(i,n_i-(a_j-b_j))}\circ \widetilde{\widetilde{\mu}}_j'$ if $\gamma_j = i$ and $\widetilde{\widetilde{\mu}}_j=\widetilde{\widetilde{\mu}}_j'$ if $i_j \neq i$.
		Let $k_1<\dots<k_{n_i-b-1}= \ell$ denote the positions in $\gamma'$ such that $\gamma_{k_l} = i=i_\ell$. It suffices to argue that we can apply commutations to bring the $n_i-b-1$ mutations from the first factor in (\ref{eqn:Menardfactoring}) adjacent to its respective instance of $\widetilde{\widetilde{\mu}}_{k_l}'$.
		
		From the description of M\'enard's algorithm for open Richardson seeds in \citep[\S5]{BY25}, after the application of $\widetilde{\widetilde{\mu}}_{k_l}'\circ \cdots \circ \widetilde{\widetilde{\mu}}_2'\circ \widetilde{\widetilde{\mu}}_1'$ to $\seed(\underline{w})$, we never mutate at any vertex adjacent to $(i,n_i-k_l+1)$ in the rest of the sequence $M(\gamma')$.
		This means the star neighbourhood of $(i,n_i-k_l+1)$ is preserved, and we can iteratively apply the required commutations.
	\end{proof}
	Lemma \ref{lem:Menardfactoring} implies that
	M\'enard's sequence $M(\gamma)$ applied to $\seed(\underline{w})$ yields the same seed as the application of the sequence $M'(\gamma)= \widehat{\widehat{\mu}}_\ell \circ \widehat{\widehat{\mu}}_{\ell-1}\circ \cdots \circ \widehat{\widehat{\mu}}_1$, where
	\begin{equation}
		\widehat{\widehat{\mu}}_j = \begin{cases}
			\mu_{(i,b_j+1)}\circ\cdots \circ \mu_{(i,a_j-1)} & \text{ if } 1<a_j,\\
			\mathrm{id} & \text{ if } a_j=1.
		\end{cases}
	\end{equation}
	
	\begin{prop}\label{prop:MutSequenceAgrees}
		By cancelling repeated mutations and commuting non-adjacent mutations, we have
		\begin{equation}\label{eqn:mutationsequencerelated}
			\overleftarrow{\mu_{\gamma^+}}\circ M(\gamma) (\seed(\underline{w})) =\overrightarrow{\mu_\gamma} (\seed(\underline{w}))
		\end{equation}
	\end{prop}
	\begin{proof}
		We use induction on $\ell(\underline{w})$.
		When $\underline{w}= (i_1)$, both sequences are identity, regardless of $\gamma$. Similarly, when $\gamma=\underline{w}$, then $M(\gamma) = \overrightarrow{\mu_{\underline{w}}}$ and $\overleftarrow{\mu_{\gamma^+}} = \overleftarrow{\mu_{\underline{w}}}$, and both seeds in (\ref{eqn:mutationsequencerelated}) are equal to $\seed(\underline{w})$.
		Now suppose for induction that $\overrightarrow{\mu_{\gamma'}} (\seed(\underline{w}')) = \overleftarrow{\mu_{(\gamma')^+}} \circ M(\gamma')(\seed(\underline{w}'))$ for $\gamma'\in \Upsilon_{\underline{w}'}$ with $\ell(\underline{w}')=\ell-1$. Let $\gamma = (\gamma',\gamma_{\ell})$ be one of two subexpressions of $\underline{w}=(\underline{w}',i_\ell)$.
		We handle the two cases of $\gamma_{\ell} = i_\ell$ and $\gamma_{\ell}=-i_\ell$ separately.
		
		In the first case, $\gamma_{\ell} = i_\ell$ implies $\overrightarrow{\mu_{\gamma}} =\overrightarrow{\mu_{\gamma'}}$. Applying Lemma \ref{lem:reversingmutationsequences} and Lemma \ref{lem:Menardfactoring}, we have
		\begin{align*}
			\overrightarrow{\mu_{\gamma}} (\seed(\underline{w}))
			&= \overleftarrow{\mu_{(\gamma')^+}}\circ M(\gamma')(\seed(\underline{w}))\\
			&= \left(\overleftarrow{\mu_{\gamma^+}}\circ \overrightarrow{\mu_{\gamma^+}}\right)\circ\overleftarrow{\mu_{(\gamma')^+}}\circ M(\gamma')(\seed(\underline{w}))\\
			&= \overleftarrow{\mu_{\gamma^+}}\circ \overrightarrow{\mu_{\gamma^+}}\circ\overleftarrow{\mu_{(\gamma[\ell])^+}}\circ M(\gamma')(\seed(\underline{w}))\\
			&= \overleftarrow{\mu_{\gamma^+}}\circ \left(\mu_{(i,b_{\ell}+1)}\circ\dots\circ\mu_{(i,n_i-1)}\right)\circ M(\gamma')(\seed(\underline{w}))\\
			&= \overleftarrow{\mu_{\gamma^+}}\circ M(\gamma)(\seed(\underline{w})).
		\end{align*}
		
		In the second case, when $\gamma_{\ell} = - i_\ell$, Definition \ref{defn:rhomutationsequence} gives
		\begin{align*}
			\overrightarrow{\mu_{\gamma}} (\seed(\underline{w}))&= \left(\overleftarrow{\mu_{(\gamma')^+}} \circ \overrightarrow{\mu_{((\gamma')^+,i_\ell)}}\right)\circ\overrightarrow{\mu_{\gamma'}}(\seed(\underline{w}))\\
			&= \left(\overleftarrow{\mu_{(\gamma')^+}} \circ \overrightarrow{\mu_{((\gamma')^+,i_\ell)}}\right)\circ\left(\overleftarrow{\mu_{(\gamma')^+}}\circ M(\gamma')\right)(\seed(\underline{w}))\\
			&= \overleftarrow{\mu_{\gamma^+}}\circ \left(\mu_{(i,b_{\ell}+1)}\circ\dots\circ\mu_{(i,n_i-1)}\right)\circ M(\gamma')(\seed(\underline{w}))\\
			&= \overleftarrow{\mu_{\gamma^+}}\circ M(\gamma).\qedhere
		\end{align*}
	\end{proof}
	
	In \citep[\S5]{BY25}, M\'enard's algorithm is geometrically interpreted as iteratively viewing 
	$\mathring{Z}_{\underline{w}}^{v}$ as a closed subvariety of $\mathring{Z}_{\underline{w}}^{vs_i}$ for $vs_i<v$.	
	The coincidence of the sequences means our procedure gives another geometric interpretation for open Richardson seeds.
	\begin{cor}
		The seeds for the open Richardson varieties obtained by M\'enard's algorithm \citep{M22} are obtained by iterated pullbacks, quasi-equivalences, and defrostings of $\overrightarrow{\mu_{\gamma^+}}(\seed(\gamma^+))$ under the sequence of rational maps $\cO^{\underline{w}} \dashrightarrow \cO^{\gamma_{(1)}^+}\dashrightarrow \cO^{\gamma_{(2)}^+} \dashrightarrow\cdots \dashrightarrow  \cO^{\gamma^+}$,
		followed by the freezing and deletion of the remaining copy of $\overrightarrow{\mu_{\gamma^+}}(\seed(\gamma^+))$.
	\end{cor}
	
	\begin{rem}	
		In \citep{M22,BY25}, the sequence $M(\gamma)$ of mutations was used in the case that $\gamma$ was the left-most subexpression of a reduced word $\gamma^+=v\leq w$. The mutation sequence is followed by freezing and deletions of (a mutation of) this copy of $\overrightarrow{\mu_{\gamma^+}}(\seed(\gamma^+))$ to obtain seeds for the open Richardson variety $\mathring{Z}_{\underline{w}}^v$. 
		In contrast, our seed $\seed(\gamma)$ for the chart $\cO^\gamma$ is obtained by the freezing (and isolation via quasi-equivalence) of the vertices in the complement of $\seed(\gamma^+)$.
		We expect this to have some interpretation as some form of semi-orthogonal decomposition in the cluster categorification of $\C[\cO^{\underline{w}}]$ via preprojective algebra modules \citep{GLS11,Lec16} and hope to return to this in future work.
		
		It is unclear to the author whether there is a reasonable categorification of $\C[\cO^{\gamma}]$ such that its maps to $\C[\cO^{\underline{w}}]$ and from $\C[\cO^{\gamma^+}]$ can be represented by appropriate functors. The isolated frozens in the seeds for $\cO^{\gamma^+}$ should correspond to indecomposable projectives with no morphisms to the other summands of the cluster tilting object.
	\end{rem}

	\bibliographystyle{plainnat}
	\bibliography{citelist}
\end{document}